\documentclass[12pt]{amsart}
\usepackage{amsfonts}
\usepackage{amsfonts,latexsym,rawfonts,amsmath,amssymb,amsthm}
\usepackage[plainpages=false]{hyperref}
\usepackage{graphicx}
\UseRawInputEncoding
\numberwithin{equation}{section}

\RequirePackage{color}
\theoremstyle{plain}

\newtheorem{theorem}{Theorem}[section]

\newtheorem{corollary}{Corollary}[section]

\newtheorem{lemma}[theorem]{Lemma}

\newtheorem{pro}[theorem]{Problem}

\newcommand{\beq}{\begin{equation}}
\newcommand{\eeq}{\end{equation}}
\newcommand{\beqs}{\begin{eqnarray*}}
\newcommand{\eeqs}{\end{eqnarray*}}
\newcommand{\beqn}{\begin{eqnarray}}
\newcommand{\eeqn}{\end{eqnarray}}
\newcommand{\beqa}{\begin{array}}
\newcommand{\eeqa}{\end{array}}

\def\phi{\varphi}

\numberwithin{equation}{section}

\begin{document}
\setlength{\baselineskip}{1.2\baselineskip}
\title[The $L_p$ dual Minkowski problem]
{Regularities for solutions to the $L_p$ dual Minkowski problem for unbounded closed sets}

\author{Li Chen}
\address{Faculty of Mathematics and Statistics, Hubei Key Laboratory of Applied Mathematics, Hubei University,  Wuhan 430062, P.R. China}
\email{chernli@163.com}

\author{Qiang Tu}
\address{Faculty of Mathematics and Statistics, Hubei Key Laboratory of Applied Mathematics, Hubei University,  Wuhan 430062, P.R. China}
\email{qiangtu@hubu.edu.cn}



\date{}
\begin{abstract}

Recently, the $L_p$ dual Minkowski problem for unbounded closed convex sets in a pointed closed convex
cone was proposed and a weak solution to this problem was provided.
In smooth setting, this problem is equivalent to
solving the Dirichlet problem for a class of Monge-Amp\`ere type equations.

In this paper, we show the existence, regularity and uniqueness of solutions to
this Monge-Amp\`ere type equation in the case $p\geq 1$ by studying variational properties for a family of Monge-Amp\`ere functionals.
Moreover, the existence and optimal global H\"older regularity in the case $p<1$ and $q\geq n$ is also be discussed.
\end{abstract}

\maketitle {\it \small{{\bf Keywords}: The $L_p$ dual Minkowski problem, Monge-Amp\`ere type equations, $C$-close sets.}}


\section{Introduction}

The main purpose of this paper is to study the $L_p$ dual Minkowski problem
for unbounded convex sets in views of PDEs. Such type of problem is an analogue
of the classical Minkowski type problem
concerning convex bodies (compact convex sets with nonempty interiors)
which has a long history and strong influence in convex geometry and PDEs.
Examples of the Minkowski type problem concerning convex bodies
include the classical Minkowski problem \cite{Sch13},
the $L_p$ Minkowski problem \cite{LE1}, the dual Minkowski problem \cite{Huang16},
the $L_p$ dual Minkowski problem \cite{LYZ-18} and so on.

The Minkowski type problem related to unbounded convex sets has also been studied by
Chou-Wang \cite{CW95}, Pogorelov \cite{Po80} and Urbas \cite{Ur84} for
unbounded, complete and convex hypersurfaces two decades ago. An $L_p$ version can be found
in \cite{HL21} by Huang-Liu. Recently, Schneider \cite{Sch-18, Sch-21} proposed
the Minkowski problem for unbounded closed convex set in a closed convex cone.
Soon, the corresponding $L_p$ Minkowski problem, dual Minkowski problem
and $L_p$ dual Minkowski problem were proposed by Yang-Ye-Zhu \cite{YYZ}, Li-Ye-Zhu \cite{LYZ}
and Ai-Yang-Ye \cite{AYY} respectively.

In the smooth setting, the $L_p$ dual Minkowski problem for unbounded closed convex set
in a closed convex cone \cite{AYY} is equivalent to solving the Dirichlet problem
of the Monge-Amp\`ere type equation \begin{equation}\label{Eq}
\left\{
\begin{aligned}
&(-h)^{1-p}\det (\nabla^2h+h I)=f[|\nabla h|^2+h^2]^{\frac{n-q}{2}} \quad \mbox{in} \quad \Omega,&\\
&h=0 \quad \mbox{on} \quad \partial \Omega,
\end{aligned}
\right.
\end{equation}
where $\Omega$ is an open convex set in $S^{n-1}$,
$f$ is a positive smooth function on $\overline{\Omega}$, $h$ is the unknown function, $I$
is the identity matrix, $\nabla h$ and $\nabla^2h$ are the gradient and the Hessian of $h$ on $S^{n-1}$.
A weak solution to the Dirichlet problem \eqref{Eq} was provided in \cite{AYY}.
Thus, it is interesting to study the regularities
of solutions to the Dirichlet problem \eqref{Eq}.

In order to study the regularities, it is convenient to express
the equation \eqref{Eq} in Euclidean space.
According to Lemma \ref{app-lem1}, the problem \eqref{Eq} is equivalent to the following Dirichlet problem for the Monge-Amp\`ere type equation in Euclidean space
\begin{equation}\label{Eq-transfer-1}
\left\{
\begin{aligned}
&\det (D^2 u)=g(x)(-u)^{p-1}\left[|D u|^2+(x\cdot Du-u)^2\right]^{\frac{n-q}{2}}\quad \mbox{in} \quad U, &\\
&u=0 \quad \mbox{on} \quad \partial U,
\end{aligned}
\right.
\end{equation}
where $U$ is an open convex set in $\mathbb{R}^{n-1}$,
$g$ is a positive smooth function on $\overline{U}$ (see \eqref{gg}),
$\nabla u$ and $\nabla^2 u$ are the gradient and the Hessian of $u$ on $\mathbb{R}^{n-1}$.

The problem \eqref{Eq-transfer-1} is a special case of
the following Dirichlet problem for the Monge-Amp\`ere equation which has been widely studied,
\begin{equation}\label{Eq-transfer-2}
\left\{
\begin{aligned}
&\det (D^2 u)=F(x, u, Du) \quad \mbox{in} \quad U \subset \mathbb{R}^{n},&\\
&u=0 \quad \mbox{on} \quad \partial U.
\end{aligned}
\right.
\end{equation}
The equation \eqref{Eq-transfer-2} was first studied by Pogorelov in \cite{P071}. When $F$ is independent of $Du$ and $F_u>0$,
Cheng-Yau obtained the existence and uniqueness of solutions to the equation \eqref{Eq-transfer-2} in \cite{Ch-Yau-77}. Then, Caffarelli-Nirenberg-Spruck \cite{CNS1} and Krylov \cite{Kr82} obtained the smoothness of solution for the
equation \eqref{Eq-transfer-2} up to the boundary under further regularity conditions for $F$. When $F_u$ is not necessarily positive, Caffarelli-Nirenberg-Spruck \cite{CNS1} solved the equation \eqref{Eq-transfer-2} under the assumption of the existence of a subsolution.
However, constructing such a subsolution is a difficult task. A different approach without constructing a subsolution was taken by
Tso \cite{Tso90}. He used a variational approach for a family of Monge-Amp\`ere functionals, which was introduced by Bakelman in \cite{Ba-61, Ba-83}, to study such problems. Recently, the analogous variational approach was introduced by Tong-Yau \cite{TY} to study
the solvability of the Dirichlet problem
\begin{equation*}\label{Eq-transfer-3}
\left\{
\begin{aligned}
&\det (D^2 u)=\lambda (u^{\ast})^{-k} (-u)^{l} \quad \mbox{in} \quad U, &\\
&u=0 \quad \mbox{on} \quad \partial U,
\end{aligned}
\right.
\end{equation*}
where $\lambda \in \mathbb{R}$, $k>0, l\geq 0$ and $u^\ast=x\cdot Du-u$.

Following the idea of \cite{Tso90, TY}, we wish to find a variational structure for 
the Dirichlet problem \eqref{Eq}, and use this to undertake a variational study for \eqref{Eq}.
An important ingredient in our variational approach is a Sobolev type inequality for $q$-volume (see Lemma \ref{Vq}).
Let $\Omega$ be an open set with the smooth boundary in $S^{n-1}$, we call $\Omega$ is strictly
convex domain in $S^{n-1}$ if the cone $\widehat{\Omega}=\{\lambda x \mid x \in \Omega, \lambda>0\}$
is a strictly convex domain in $\mathbb{R}^{n}$. The following is our first main result.

\begin{theorem}\label{th-main-1}
Let $\Omega$ be an open, bounded, smooth and strictly convex domain in $S^{n-1}$, $f$ be a positive smooth function on $\overline{\Omega}$
and $p\geq 1$.
\begin{enumerate}
\item[{\em(\romannumeral1)}] If  $q>p$, then there exists a unique and non-zero solution  $h \in C^{\infty}(\overline{\Omega})$
to the Dirichlet problem \eqref{Eq}.
\item [{\em(\romannumeral2)}] If $p=q$, then there exists a unique and non-zero $\lambda$ such that the Dirichlet problem \eqref{Eq}  with $f$ replaced by $\lambda f$ admits a non-zero solution $h\in C^{\infty}(\overline{\Omega})$. Moreover, the solution is unique up to scaling by a positive constant.
\item [{\em(\romannumeral3)}] If $p>q\geq n$, then there exists a non-zero solution  $h \in C^{\infty}(\overline{\Omega})$ to the Dirichlet problem \eqref{Eq}.
\end{enumerate}
\end{theorem}

The existence and uniqueness of smooth solutions to
the $L_p$ dual Minkowski problem \eqref{Lp dual} for convex bodies
have been proved in \cite{HZ} for $p>q$ and in \cite{Chen-L19} for $p=q\neq 0$.
For the other case $p<q$, the uniqueness may fails \cite{CCL21, LLL22, JWW21, JWW23}.
Thus, although the Monge-Amp\`ere equations \eqref{Eq} and \eqref{Lp dual} differ from
each other only by a negative sign, their solvability seems to be quite different.

It would be desirable to obtain the existence of solutions to the Dirichlet problem \eqref{Eq} in the case $p<1$, but we have not been able to do this by the variational approach. In fact, the equation \eqref{Eq} becomes a singular Monge-Amp\`ere function
in the case $p<1$, and the high order regularity of solutions to the equation \eqref{Eq} may fail up to boundary. In details, we get the following result.

\begin{theorem}\label{th-main-2}
Assume $p<1$ and $q\geq n\geq 3$. Let $\Omega$ be an open, bounded, smooth and convex domain in $S^{n-1}$, $f\in C^{\infty}(\Omega) \cap C(\overline{\Omega})$ with $f>0$.
Then  there exist a unique nontrivial solution  $h\in C^{\infty}(\Omega)\cap C^{\frac{q-n+2}{q-p}}(\overline{\Omega})$ to the equation \eqref{Eq} with the following estimate
\begin{equation}
|h(x)| \leq C(n,p,q, \mbox{diam}(\Omega), \sup f) [\mbox{dist}(x, \partial \Omega)]^{\frac{q-n+2}{q-p}}
\end{equation}
for any $x\in \Omega$. Moreover, the exponent $\frac{q-n+2}{q-p}$ is optimal, i.e., for any $a\in (\frac{q-n+2}{q-p}, 1)$, there exist a bounded convex domain $\Omega\subset S^{n-1}$ such that
the solution $h$ of the equation \eqref{Eq} satisfies $h\notin C^a(\overline{\Omega})$.
\end{theorem}

The ideas for the proof of the above result comes from the study for the following Dirichlet problem
\begin{equation}\label{Eq-transfer-4}
\left\{
\begin{aligned}
&\det (D^2 u)= (u^{\ast})^{-k} (-u)^{l} \quad \mbox{in} \quad U\subset \mathbb{R}^n,&\\
&u=0 \quad \mbox{on} \quad \partial U,
\end{aligned}
\right.
\end{equation}
where $l<0$ and $k>0$. When  $k=0$ and $l=-n-2$,  Cheng-Yau \cite{Ch-Yau-77} obtained the existence result for the equation \eqref{Eq-transfer-4}. Then Le \cite{LNQ-1} extended the existence result to the case $l<0$.  When $l=-n-k-2$, the equation \eqref{Eq-transfer-4} was related to proper affine hyperspheres and Chen-Huang \cite{Chen-L19} showed the existence of solutions
to the equation \eqref{Eq-transfer-4} in the space $C^{\infty}(\Omega)\cap C(\overline{\Omega})$ via the
regularization method. Moreover, Le \cite{LNQ-1, LNQ-2} established the optimal global H\"older regularity of solutions.

The rest of the paper is organized as follows. In Section 2, we start with some
preliminaries. The proofs of Theorem \ref{th-main-1}  are given in section 3.
In section 4, the existence result and optimal global H\"older regularity for solutions to the equation \eqref{Eq}
in the case $p<1$ and $q\geq n$ are established.
In the appendix, we establish some basic a priori estimates for the elliptic and parabolic Monge-Amp\`ere equations.

\section{Preliminaries}

In this section, we collect the necessary background, preliminaries, and notations.
More details can be found in \cite{Huang16, LYZ-18, Ur-1}
for convex bodies and in \cite{YYZ, LYZ, Sch-18} for $C$-close convex sets.

\subsection{Convex bodies and their associated $L_p$ dual Minkowski problem}

Let $\mathbb{R}^n$ be the $n$-dimensional Euclidean space. The unit sphere in $\mathbb{R}^n$ is denoted by
$S^{n-1}$. A convex body in $\mathbb{R}^n$ is a compact convex set with nonempty interior. Denote
by $\mathcal{K}^{n}_{0}$ the class of convex bodies in $\mathbb{R}^n$ that contain the origin in their interiors.
The support function $h: \mathbb{R}^n\rightarrow \mathbb{R}$ of a convex body $K$ is defined as
\begin{equation*}
h_K(x)=\max \{ x \cdot y: y \in K\},
\end{equation*}
where $\cdot$ is the standard inner product in $\mathbb{R}^n$.
The radial function $\rho$ of $ K \in \mathcal{K}^{n}_{0}$ is defined as
\begin{equation*}
\rho_K(u)=\max\{\lambda>0 : \lambda u\in K\}.
\end{equation*}

For a convex body $K$, its $L_p$ surface area measure
$S_p(K, \cdot)$ is defined by Lutwak \cite{LE1},
\begin{eqnarray*}
S_{p}(K, \omega)=\int_{\nu_{K}^{-1}(\omega)} (x \cdot \nu_{K}(x))^{1-p}dx
\end{eqnarray*}
for any Borel set $\omega\subset S^{n-1}$, where the set $\nu_{K}^{-1}(\omega)$ is the inverse image
of $\omega$ under the Gauss map $\nu_K$ of $K$.
If $p=1$, it is just the surface area measure of $K$. Recently, Huang-LYZ in \cite{Huang16} proposed a fundamental family
of geometric measures in the dual Brunn-Minkowski theory: the dual curvature measure which is defined by
\begin{eqnarray*}
\widetilde{C}_{q}(K, \omega)=\frac{1}{n} \int_{\alpha_{K}^{\ast}(\omega)} h_K^{-p}(\alpha_{K}(u)) \rho_K^q(u) du
\end{eqnarray*}
for any Borel set $\omega \subset S^{n-1}$, where $\alpha_{K}^{\ast}(\omega)$ is
the radial Gauss image of $K$ given by
\begin{eqnarray*}
\alpha_{K}^{\ast}(\omega)=\{u \in S^{n-1}: u\rho_K(u) \in \nu_{K}^{-1}(\omega)\}.
\end{eqnarray*}
Later, LYZ in \cite{LYZ-18} unified the $L_p$ surface area measure and the dual curvature measure by introducing
the $L_p$ dual curvature measure
\begin{eqnarray*}
\widetilde{C}_{p, q}(K, \omega)=\frac{1}{n} \int_{\alpha_{K}^{\ast}(\omega)} h_K^{-p}(\alpha_{K}(u)) \rho_K^q(u) du
\end{eqnarray*}
for any Borel set $\omega \subset S^{n-1}$. It is worth pointing out that
the $L_p$-dual curvature measure becomes the $L_p$ surface area measure for $q=n$ and the dual curvature measure for $p=0$.

The following $L_p$ dual Minkowski problem was posed in \cite{LYZ-18}.
\begin{pro}\label{problem-L-1}
For $p, q \in \mathbb{R}$, under what conditions on a non-zero finite Borel measure $\mu$ defined on $S^{n-1}$, can one
find $K \in \mathcal{K}^{n}_{0}$ such that
$$\mu=\widetilde{C}_{p, q}(K, \cdot)?$$
\end{pro}

The $L_p$ dual Minkowski problem becomes the $L_p$ Minkowski problem for $q=n$ \cite{LE1} and
the dual Minkowski problem for $p=0$ \cite{Huang16}.
When the given measure $\mu$ has a density $f$, the $L_p$ dual Minkowski problem
is equivalent to solving the following Monge-Amp\`ere type equation on $S^{n-1}$:
\begin{equation}\label{Lp dual}
h^{1-p}\det (\nabla^2 h+h I)=f[|\nabla h|^2+h^2]^{\frac{n-q}{2}},
\end{equation}
where $f$ is a smooth function on $S^{n-1}$, $h$ is the unknown function, $I$
is the identity matrix, $\nabla h$ and $\nabla^2h$ are the gradient and the Hessian of $h$ on $S^{n-1}$.

\subsection{$C$-close convex sets and their associated $L_p$ dual Minkowski problem}

A set $C\subseteq \mathbb{R}^n$ is said to be a closed convex cone, if $C$ is closed and convex such that
the interior of $C$ is nonempty and $\lambda x \in C$ for all $x \in C$ and $\lambda\geq 0$. If
$C\cap \{-x: x \in C\}=\{o\}$, then the closed convex cone $C$ is called a pointed cone.
For a pointed closed convex cone $C$, its polar cone is denoted by $C^{\circ}$ and defined by
\begin{eqnarray*}
C^{\circ}=\{ x \in \mathbb{R}^n:  x \cdot y \leq 0 \ \mbox{for all} \ y \in C\}.
\end{eqnarray*}

Let $C$ be a pointed closed convex cone with nonempty interior and
$A=C \backslash \mathbb{A}$ for any $\mathbb{A}\subsetneq C$. For a closed convex set $\mathbb{A}\subsetneq  C$, if
$0<V_n(A)<\infty$, we call $\mathbb{A}$ a $C$-close set and $A$ a $C$-coconvex set, while if $A$ is bounded and nonempty,
we call $\mathbb{A}$ a $C$-full set. Note that $o \notin \mathbb{A}$ if $\mathbb{A}$ is $C$-close or $C$-full.

Most concepts for convex bodies can be defined for $C$-close set (with slight or without changes).
For example, the support function of a $C$-close set $\mathbb{A}$ can be defined by
\begin{equation}
h_C(\mathbb{A}, x)=\sup \{ x \cdot y: y \in \mathbb{A}\}, \quad x\in \Omega_{C^\circ},
\end{equation}
where $\Omega_{C^{\circ}}=S^{n-1}\cap \mathrm{int} C^{\circ}$. Note that
$o \not\in \mathbb{A}$ and hence $-\infty<h_C(\mathbb{A}, x)<0$ for any $x \in \Omega_{C^{\circ}}$.
Let $\Omega_{C}=S^{n-1}\cap \mathrm{int} C$.
The radial function of $\mathbb{A}$ is defined by
\begin{eqnarray*}
\rho_C(\mathbb{A}, u)=\sup\{r>0 : r u \in C\backslash \mathbb{A} \}, \quad u \in \Omega_C.
\end{eqnarray*}
At $u\in \Omega_{C}$, $\rho_C(\mathbb{A}, u)$ could be finite or $\infty$ depending on whether $\mathbb{A}$ intersects with $\partial C$ at the direction $u$.

\begin{lemma}
Assume that $\mathbb{A}$ is a $C$-full set, we have
\begin{eqnarray}\label{hro}
\max_{\Omega_{C^\circ}}|h_{C}(\mathbb{A}, \cdot)|=\min_{\Omega_{C}} \rho_C(\mathbb{A}, \cdot).
\end{eqnarray}
\end{lemma}

\begin{proof}
Assume that
\begin{eqnarray*}
\rho_C(\mathbb{A}, u_{min})=\min_{\Omega_C} \rho_C(\mathbb{A}, \cdot),
\quad h_{C}(\mathbb{A}, x_{min})=\min_{\Omega_{C^\circ}}h_{C}(\mathbb{A}, \cdot).
\end{eqnarray*}
On one hand, by the definition, we have
\begin{eqnarray*}
h_{C}(\mathbb{A}, x_{min})\geq \rho_C(\mathbb{A}, u_{min})u_{min} \cdot x_{min}\geq-\rho_C(\mathbb{A}, u_{min}).
\end{eqnarray*}
Thus,
\begin{eqnarray*}
\max_{\Omega_{C^\circ}}|h_{C}(\mathbb{A}, \cdot)|\leq\min_{\Omega_{C}} \rho_C(\mathbb{A}, \cdot).
\end{eqnarray*}
On the other hand, we have
\begin{eqnarray*}
|h_{C}(\mathbb{A}, x_{min})|\geq \rho_C(\mathbb{A}, x_{min})\geq \min_{\Omega_{C}} \rho_C(\mathbb{A}, \cdot).
\end{eqnarray*}
So, we complete the proof.
\end{proof}

Let $\mathbb{A}$ a $C$-close set. Assume $\partial \mathbb{A}\subsetneq \mathrm{int} C$ is
smooth and strictly convex with $\lim_{x\rightarrow \partial\Omega_{C^\circ}}h_{C}(\mathbb{A}, x)=0$.
Clearly, $h_{C}(\mathbb{A}, x) \in (-\infty, 0)$ for all $x \in \Omega_{C^\circ}$ due to
$o \not\in \mathbb{A}$. In this case, $\partial \mathbb{A}$ can be determined by its radical function
$\rho_{C}(\mathbb{A}, \cdot)$. If $x \in \Omega_{C^\circ}$ is the outer normal of $\partial \mathbb{A}$ at
the point $u \in \partial \mathbb{A}$, then $u=h_{C}(\mathbb{A}, x)x+\nabla h_{C}(\mathbb{A}, x)$. This gives
\begin{eqnarray}\label{LYZ1.4}
\rho_C(\mathbb{A}, u)=\sqrt{|h_{C}(\mathbb{A}, x)|^2+|\nabla h_{C}(\mathbb{A}, x)|^2}.
\end{eqnarray}
Moreover,
\begin{eqnarray}\label{LYZ1.5}
d u=\rho_C^{-n}(-h_{C})\mathrm{det}(\nabla^2 h_{C}+ h_{C}I)dx.
\end{eqnarray}

Inspired by the $L_p$ dual curvature measure introduced by LYZ for convex bodies \cite{LYZ-18}, Ai-Yang-Ye \cite{AYY} introduce the $L_p$ dual curvature measure for a $C$-close set $\mathbb{A}$
\begin{equation*}
\widetilde{C}_{p, q}(\mathbb{A}, \omega)=\frac{1}{n} \int_{\alpha_{\mathbb{A}}^{\ast}(\omega)} h_C^{-p}(\mathbb{A}, \alpha_{\mathbb{A}}(u)) \rho_C^q(\mathbb{A}, u) du,
\end{equation*}
where $\omega $ is a Borel set in $\Omega_{C^{\circ}}$ and $\alpha_{\mathbb{A}}^{\ast}(\cdot)$ is the reverse radial Gauss image of $\mathbb{A}$ (see (4.6) in \cite{LYZ} for the definition).
It is worth pointing out that $\widetilde{C}_{p, q}(\mathbb{A}, \cdot)$ is
the $L_p$ surface area measure for $q=n$ \cite{YYZ}
and the $q$-th dual curvature measure for $p=0$ \cite{LYZ}.
Thus, the following $L_p$ dual Minkowski problem for $C$-close sets is proposed in \cite{AYY}.
\begin{pro}\label{problem-1}
For $p, q \in \mathbb{R}$, under what conditions on a nonzero finite Borel measure $\mu$ defined on $\Omega_{C^\circ}$, can one find a $C$-close set $\mathbb{A}$ such that
$$\mu=\widetilde{C}_{p, q}(\mathbb{A}, \cdot)?$$
\end{pro}

Obviously, Problem \ref{problem-1} unifies the $L_p$ Minkowski problem for $C$-close sets \cite{YYZ} and the dual Minkowski problem for $C$-close sets  \cite{LYZ}. In particular, when the given measure $\mu$ has a density $f$, Problem \ref{problem-1}
is equivalent to solving the Dirichlet problem \eqref{Eq} with $\Omega=\Omega_{C^\circ}$.

\subsection{The $q$-volume functional and the Sobolev type inequality}

If $\mathbb{A}$ is a $C$-full set, then
$0<\rho_C(\mathbb{A}, u)<\infty$ for any $u \in \mathbb{S}^{n-1} \cap \partial C$. The $q$-volume
of $C \backslash \mathbb{A}$ or $A$ is defined by
\begin{equation}\label{con}
V_q(C \backslash \mathbb{A})=\frac{1}{q} \int_{\Omega_C} \rho^q_C(\mathbb{A}, u) du.
\end{equation}
When $q=n$, it is just the volume of $A$. Moreover,
if $\partial \mathbb{A}$ is a smooth hypersurface, using \eqref{LYZ1.4} and \eqref{LYZ1.5}, we have
\begin{equation}\label{con-1}
V_q(C \backslash \mathbb{A})=\frac{1}{q}\int_{\Omega_{C^\circ}}(|h_{C}|^2+|\nabla h_{C}|^2)^{\frac{q-n}{2}}(-h_C)\det (\nabla^2 h_C+h_C I)dx.
\end{equation}
Thus, we define the $q$-volume functional with respect with $h_C$
\begin{equation}\label{q-V}
V_q(h_C)=\frac{1}{q}\int_{\Omega_{C^\circ}}(|h_{C}|^2+|\nabla h_{C}|^2)^{\frac{q-n}{2}}(-h_C)\det (\nabla^2 h_C+h_C I)dx.
\end{equation}
Now, we will calculate the first variation of $V_q$ with respect with $h_C$.
For convenience, we denote by $\Omega=\Omega_{C^\circ}$,
$h(x)=h_C(\mathbb{A}, x)$ and $\rho(u)=\rho_C(\mathbb{A}, u)$.

\begin{lemma}
Let $\mathbb{A}_t$ be a family of $C$-full sets with the support function $h(\cdot, t)$ satisfying
$h(\partial \Omega, t)=0$. We denote by $h=h(\cdot, 0)$ and $\varphi=\frac{d}{d t}\Big|_{t=0}h(\cdot, t)$.
Then, the first variation of $V_q$ at $h$ with respect to
$\phi$ is given by
\begin{eqnarray}\label{Var}
q\cdot\delta V_q(h)[\phi]=-q\int_{\Omega}\rho^{q-n}\phi\det (\nabla^2 h+h I) \ dx,
\end{eqnarray}
where $\rho=\sqrt{h^2+|\nabla h|^2}$.
\end{lemma}

\begin{proof}
For convenience, we denote by $b=\nabla^2 h+h I$, $b_{ij}=h_{ij}+h\delta_{ij}$ and
$(b^{ij})=(b_{ij})^{-1}$. Then, using $(\log \mathrm{det}\ b)_k=b^{ij}b_{ij k}$,
the first variation of $V_q$ at $h$ with respect to
$\phi$ is given by
\begin{eqnarray*}
&&q\cdot\delta V_q(h)[\phi]
\\&=&q\frac{d}{dt}\bigg|_{t=0}V_q(h(\cdot, t))
\\&=&-\int_{\Omega}\rho^{q-n}\phi\det b \ dx
-(q-n)\int_{\Omega}\rho^{q-n-2}(\nabla h \cdot \nabla \varphi +h\phi)h\det b \ dx
\\&&+\int_{\Omega}\rho^{q-n}(-h) b^{ij}(\varphi_{ij}+\varphi \delta_{ij})\det b \ dx.
\end{eqnarray*}
Then, using the fact $(b^{ij}\det b)_j=0$ and $h|_{\partial \Omega}=0$, and integrating by parts give
\begin{eqnarray*}
&&q\cdot\delta V_q(h)[\phi]
\\&=&-\int_{\Omega}\rho^{q-n}\phi\det b \ dx-(q-n)\int_{\Omega}\rho^{q-n-2}(\nabla h \cdot \nabla \varphi +h\phi) h \det b \ dx
\\&&+\int_{\Omega}\rho^{q-n}(-h) \varphi b^{ij} \delta_{ij}\det b \ dx
+\int_{\Omega}\rho^{q-n}h_j \varphi_{i} b^{ij}\det b \ dx
\\&&+(q-n)\int_{\Omega}h \rho^{q-n-1}\rho_j \varphi_{i} b^{ij}\det b \ dx.
\end{eqnarray*}
Then, using $(b^{ij}\det b)_j=0$ again and
\begin{equation*}
b^{ij}\rho_j\varphi_i=\frac{1}{\rho}\nabla h \cdot \nabla \varphi, \quad
b^{ij}\rho_jh_i=\frac{1}{\rho}|\nabla h|^2,
\end{equation*}
we have
\begin{eqnarray*}
&&q\cdot\delta V_q(h)[\phi]\\&=&-\int_{\Omega}\rho^{q-n}\phi\det b \ dx
-(q-n)\int_{\Omega}\rho^{q-n-2}h^2\phi \det b \ dx
\\&&+\int_{\Omega}\rho^{q-n}(-h) \varphi b^{ij} \delta_{ij}\det b \ dx
+\int_{\Omega}\rho^{q-n}h_j \varphi_{i} b^{ij}\det b \ dx
\\&=&-\int_{\Omega}\rho^{q-n}\phi\det b \ dx
-(q-n)\int_{\Omega}\rho^{q-n-2}h^2\phi \det b \ dx
\\&&+\int_{\Omega}\rho^{q-n}(-h) \varphi b^{ij} \delta_{ij}\det b \ dx
-\int_{\Omega}\rho^{q-n}h_{ij} \varphi b^{ij}\det b \ dx
\\&&-(q-n)\int_{\Omega}\rho^{q-n-1}\rho_ih_{j} \varphi b^{ij}\det b \ dx
\\&=&-\int_{\Omega}\rho^{q-n}\phi\det b \ dx-(q-n)\int_{\Omega}\rho^{q-n}\phi \det b \ dx
\\&&-(n-1)\int_{\Omega}\rho^{q-n}\varphi\det b \ dx
\\&=&-q\int_{\Omega}\rho^{q-n}\phi\det b \ dx.
\end{eqnarray*}
So, we complete the proof.
\end{proof}

In particular, for $q=0$, we have
\begin{equation*}
\delta\bigg(\int_{\Omega}\frac{(-h)\det (\nabla^2 h+h I)}{\rho^{n}}dx\bigg)=0.
\end{equation*}
Thus,
\begin{equation*}
\int_{\Omega}\frac{(-h)\det (\nabla^2 h+h I)}{\rho^{n}}dx=const.
\end{equation*}
Moreover, we have

\begin{corollary}
If $\mathbb{A}$ is a $C$-full set with the support function $h$ satisfying $h|_{\partial \Omega}=0$,
then we have
\begin{equation}\label{Sob}
\int_{\Omega}\frac{(-h)\det (\nabla^2 h+h I)}{\rho^{n}}dx=\mathrm{Area}(\Omega_C).
\end{equation}
\end{corollary}

\begin{proof}
The equality \eqref{Sob} can be easily deduced by
\eqref{con} and \eqref{con-1}
\begin{equation*}
\int_{\Omega}\frac{(-h)\det (\nabla^2 h+h I)}{\rho^{n}}dx=\int_{\Omega_C} du=\mathrm{Area}(\Omega_C).
\end{equation*}
\end{proof}

Using this corollary, we can easily deduce the following Sobolev type inequality for the $q$-volume functional.

\begin{lemma}\label{Vq}
Let $\mathbb{A}$ be a $C$-full set with the support function $h$ satisfying $h|_{\partial \Omega}=0$ and $q>0$.
Then,
\begin{equation*}
V_q(h)\geq \frac{\mathrm{Area}(\Omega_C)}{q}\parallel h \parallel_{C^{0}(\Omega)}^{q}.
\end{equation*}
Thus,
\begin{equation*}
V_q(h)\geq \frac{\mathrm{Area}(\Omega_C)}{q \mathrm{Vol}(\Omega)} \int_{\Omega} |h|^q dx.
\end{equation*}
\end{lemma}

\begin{proof}
Using \eqref{hro} and \eqref{Sob}, we have
\begin{eqnarray*}
V_q(h)&=&\frac{1}{q}\int_{\Omega}\rho^{q-n}(-h)\det (\nabla^2 h+h I)dx
\\&\geq& \frac{\mathrm{Area}(\Omega_C)}{q}(\min_{\Omega}\rho)^{q}
\\&=&\frac{\mathrm{Area}(\Omega_C)}{q}\parallel h \parallel_{C^{0}(\Omega)}^{q},
\end{eqnarray*}
as claimed.
\end{proof}

\section{The Dirichlet problem in the case $p\geq1$}

By Theorem 1.3 in \cite{Sa14} and Theorem 1.2 in \cite{LS17}, it is
possible to obtain the existence of smooth solutions to
the Dirichlet problem \eqref{Eq} for $p\geq1$. We follow
the ideas in \cite{Tso90} to find the variational functional of the Dirichlet problem \eqref{Eq}.
Then, we obtain the existence by using the corresponding parabolic gradient flow.
In fact, the parabolic gradient flow method is widely used to prove the
existence of smooth solutions to the Minkowski type problems, see
\cite{CW, CW00, BIS19, LSW20, CHZ19, Chen-L19, LL20, CLLN22, CTWX22, CWX22}
and the references therein.

The argument of the Dirichlet problem \eqref{Eq} is divided into three cases:

(1) Subcritical case: $p<q$;

(2) Supercritical case: $p>q$;

(3) Critical case: $p=q$.

In this section, let $\Omega$ be an open, bounded, smooth and strictly convex domain in $S^{n-1}$, $f\in C^{\infty}(\overline{\Omega})$ with $f>0$ and $h_0$ be the support function of a smooth $C$-full set with $h_0(\partial \Omega)=0$.

\subsection{Subcritical case}

\subsubsection{A parabolic gradient flow}

Since the original equation \eqref{Eq} becomes degenerate or singular
at the boundary, we modify the original equation \eqref{Eq} by a perturbation
\begin{equation}\label{1-Eq}
\left\{
\begin{aligned}
&[|\nabla h|^2+h^{2}]^{\frac{q-n}{2}}\det (\nabla^2 h+hI)=
(\varepsilon-h)^{p-1}f \quad \mbox{in} \quad  \Omega ,&\\
&h=0 \quad \mbox{on} \quad  \partial \Omega.
\end{aligned}
\right.
\end{equation}

The equation \eqref{1-Eq} is the Euler-Lagrange equation of the the functional
\begin{eqnarray*}
\mathcal{J}_\varepsilon(h):=V_q(h)-\frac{1}{p}\int_{\Omega}(\varepsilon-h)^p f(x)dx.
\end{eqnarray*}
This fact can be easily seen by its variation
\begin{eqnarray}\label{VJ}
\delta\mathcal{J}_{\varepsilon}(h)[\varphi]=-\int_{\Omega}\phi\Big[\rho^{q-n}\det (\nabla^2 h+h I)-(\varepsilon-h)^{p-1} f\Big]dx.
\end{eqnarray}
This variation \eqref{VJ} can be derived by \eqref{Var}.

In this subsection, we will study a gradient flow of the functional $J_\varepsilon$. In details, we consider
the following parabolic equation with initial condition $h_0$:
\begin{equation}\label{t-Eq}
\left\{
\begin{aligned}
&h_t-\log \det (\nabla^2 h+hI)+\frac{n-q}{2}\log [|\nabla h|^2+h^2]=\log[(\varepsilon-h)^{1-p}f^{-1}]
\quad \mbox{in}  \quad \Omega_T,
&\\
&h=0 \quad \mbox{on}  \quad \partial \Omega \times [0, T],
&\\&h=h_0 \quad \mbox{on}  \quad  \Omega \times \{0\},
\end{aligned}
\right.
\end{equation}
where $\Omega_T=\Omega \times (0, T]$.
By the first variation formula \eqref{VJ} of $\mathcal{J}_\varepsilon$,
we can see that

\begin{lemma}\label{De}
$\mathcal{J}_\varepsilon$ is non-increasing along this flow \eqref{t-Eq}.
\end{lemma}

\subsubsection{The long time existence}

The short time existence can be guaranted by Theorem A in \cite{Tso90}.
\begin{theorem}
There exists a unique $T^\star$, $0<T^\star\leq+\infty$, such that the
flow \eqref{t-Eq} has a unique solution $h$ which belongs to
$C^{1, 1}(\overline{\Omega_T})\cap C^{\infty}(\overline{\Omega}\times (0, T])$ for
all $T<T^\star$ and
\begin{eqnarray}\label{Sh-1}
\mathcal{J}_{\varepsilon}(h(\cdot, t))\leq\mathcal{J}_{\varepsilon}(h_0)
\end{eqnarray}
for all $t<T^\star$.
\end{theorem}

In order to get the long time existence, we first establish the a priori estimates.
In this subsection, let $T<T^\star$ and $\mathbb{A}_t$ be a family of $C$-full sets with the support function $h(\cdot, t)$
which belongs to
$C^{1, 1}(\overline{\Omega_T})\cap C^{\infty}(\overline{\Omega}\times (0, T])$
satisfying the flow \eqref{t-Eq}. Moreover,
by scaling, we choose $h_0$ such that
$\mathcal{J}_{\varepsilon}(h_0)<\mathcal{J}_{0}(h_0)<0$ for $q>p>0$.

\begin{lemma}
For $q>p\geq 1$, we have
\begin{eqnarray}\label{f-C0}
-C \leq h(x, t)< 0, \quad \forall \ (x, t) \in \Omega \times [0, T],
\end{eqnarray}
where $C$ is independent of $\varepsilon$.
Moreover,
\begin{eqnarray}\label{ro-b}
\min_{\overline{\Omega}}\sqrt{|\nabla h|^2+h^2}(\cdot, t)
=\parallel h(\cdot, t)\parallel_{C^0(\overline{\Omega})}\geq \frac{1}{C},
\quad \forall \ t \in [0, T],
\end{eqnarray}
where $C$ is independent of $\varepsilon$.
\end{lemma}

\begin{proof}
Using Lemma \ref{Vq} and \eqref{Sh-1}, we have
\begin{eqnarray}\label{JJJ}
\mathcal{J}_{\varepsilon}(h_0)\geq \mathcal{J}_{\varepsilon}(h)\geq \frac{C}{q}\parallel h \parallel_{C^0(\overline{\Omega})}^{q}-\frac{C}{p}\parallel \varepsilon-h\parallel_{C^0(\overline{\Omega})}^p.
\end{eqnarray}
Hence, we can deduce if $q>p$
\begin{eqnarray*}
\parallel h \parallel_{C^0(\overline{\Omega})}\leq C.
\end{eqnarray*}
Since $\mathcal{J}_{\varepsilon}(h_0)<\mathcal{J}_{0}(h_0)<0$, it follows that
\begin{eqnarray*}
-C\geq \mathcal{J}_{\varepsilon}(h)\geq-\frac{1}{p}\int_{\Omega}(\varepsilon-h)^pf(x)dx.
\end{eqnarray*}
Thus,
\begin{eqnarray*}
\frac{1}{p}\int_{\Omega}(\varepsilon-h)^pf(x)dx\geq C>0,
\end{eqnarray*}
which implies
\begin{eqnarray*}
\parallel h(\cdot, t)\parallel_{C^0(\overline{\Omega})}\geq \frac{1}{C}.
\end{eqnarray*}
Then, \eqref{ro-b} follows from the relation \eqref{hro}.
So we complete the proof.
\end{proof}

\begin{theorem}\label{Dh}
The flow \eqref{t-Eq} exists all the time for $q>p\geq 1$.
Moreover, after choosing a subsequence, the flow converges to a non-zero and smooth solution to
the equation \eqref{1-Eq}.
\end{theorem}

\begin{proof}
Using Lemma \ref{app-lem1}, we transform the flow \eqref{t-Eq} of $h$ to that of $u$
in $U\subset \mathbb{R}^{n-1}$

\begin{equation}\label{u-Eq}
\left\{
\begin{aligned}
&\frac{u_t}{\sqrt{1+x^2}}-\log \det (D^2 u)=-G(x, u, Du)\quad \mbox{in}  \quad U \times (0, T],
&\\
&u=0 \quad \mbox{on}  \quad \partial U \times [0, T],
&\\&u=u_0 \quad \mbox{on}  \quad  U \times \{0\},
\end{aligned}
\right.
\end{equation}
where
\begin{eqnarray*}
G(x, u, Du)=\frac{n-q}{2}\log [|D u|^2+(x \cdot Du-u)^2]+\log [(\varepsilon\sqrt{1+|x|^2}-u)^{p-1}g].
\end{eqnarray*}
Using the relation \eqref{h-u} and the inequality \eqref{ro-b}, we arrive
\begin{eqnarray}\label{D-S}
\Big[|Du|^2+(x\cdot Du-u)^2\Big](x, t)\geq \frac{1}{C}, \quad \forall \ (x, t) \in U \times [0, T].
\end{eqnarray}
Using \eqref{D-S} and the $C^0$ estimate \eqref{f-C0}, we obtain that the right part
of the equation \eqref{u-Eq} satisfies
\begin{eqnarray*}
G(x, u, Du)\leq \frac{n-1}{2}\log(1+|Du|^2)+C,
\quad \forall \ x \in \overline{U}.
\end{eqnarray*}
Then, we can use Lemma \ref{A-gra} and Lemma \ref{D2-i} in Appendix to obtain
the gradient estimates and $C^2$ estimates. Thus, we conclude that
evolution equation \eqref{PMA} is uniformly parabolic on any finite time
interval. Thus, the result of \cite{KS.IANSSM.44-1980.161} and the standard
parabolic theory show that the solution of \eqref{PMA} exists for all
time. Using these estimates again, a subsequence of $h(\cdot, t)$ converges to
a function $h_{\infty}$.
Since $\frac{d}{d t}\mathcal{J}_{\varepsilon}(h(\cdot, t))\leq0$ for any $t>0$ by Lemma \ref{De}, we have
\begin{equation*}
\int_0^t \Big[-\frac{d}{d t}\mathcal{J}_{\varepsilon}(h(\cdot, t))\Big]dt
=\mathcal{J}_{\varepsilon}(0)-\mathcal{J}_{\varepsilon}(t) \leq \mathcal{J}_{\varepsilon}(0),
\end{equation*}
which implies that there exists a subsequence of times $t_j\rightarrow\infty$ such that
\begin{equation*}
-\frac{d}{d t}\mathcal{J}_{\varepsilon}(h(\cdot, t_j)) \rightarrow 0 \quad  \mbox{ as } \ t_j\rightarrow\infty.
\end{equation*}
That is to say
\begin{eqnarray*}
0=\lim_{j\rightarrow +\infty}&&\int_{\Omega}\bigg(\rho_{j}^{q-n}\det (\nabla^2 h_{j}+h_{j} I)-(\varepsilon-h_{j})^{p-1} f\bigg)
\\&&\cdot\bigg(\log\Big[\rho_{j}^{q-n}\det (\nabla^2 h_{j}+h_{j} I)\Big]-\log\Big[(\varepsilon-h_{j})^{p-1} f \Big]\bigg)dx,
\end{eqnarray*}
where we denote $h_j(x)=h(x, t_j)$ and $\rho_j(x)=\rho(x, t_j)$. From this and the a priori estimates of $h_j$,
we conclude that $h_{\infty}$ is a smooth solution to the equation \eqref{t-Eq}.
\end{proof}

\subsubsection{The Dirichlet problem}

\begin{theorem}\label{D}
The Dirichlet problem \eqref{Eq} admits a unique and non-zero solution $h \in C^{\infty}(\overline{\Omega})$ for $q>p\geq 1$.
Moreover, this solution is the minimum of the functional
\begin{eqnarray*}
\mathcal{J}(h):=V_q(h)-\frac{1}{p}\int_{\Omega}(-h)^p f(x)dx.
\end{eqnarray*}
\end{theorem}

\begin{proof}
Theorem \ref{Dh} tells us that there exists a solution $h_\varepsilon \in C^{\infty}(\overline{\Omega})$
solving the Dirichlet problem \eqref{1-Eq}.
Using Lemma \ref{app-lem1}, we transform the Dirichlet problem \eqref{1-Eq} of $h$ to that of $u$
in $U\subset \mathbb{R}^{n-1}$. Then, there exists a solution
$u_\varepsilon \in C^{\infty}(\overline{U})$ that solves the Dirichlet problem
\begin{equation}\label{u-Eq-D}
\left\{
\begin{aligned}
&\det (D^2u_\varepsilon)=[|D u_\varepsilon|^2
+(x \cdot Du_\varepsilon-u_\varepsilon)^2]^{\frac{n-q}{2}}
(\varepsilon\sqrt{1+|x|^2}-u_\varepsilon)^{p-1}g \quad \mbox{in}  \quad U,
&\\
&u_\varepsilon=0 \quad \mbox{on}  \quad \partial U.
\end{aligned}
\right.
\end{equation}
Using \eqref{f-C0}, we have
\begin{eqnarray*}
-C \leq u_\varepsilon(x)< 0, \quad \forall \ x \in \Omega,
\end{eqnarray*}
where the constant $C$ is independent of $\varepsilon$.
Using \eqref{D-S}, the right hand term of the equation \eqref{u-Eq-D} satisfies
\begin{eqnarray*}
[|D u_\varepsilon|^2
+(x \cdot Du_\varepsilon-u_\varepsilon)^2]^{\frac{n-q}{2}}
(\varepsilon\sqrt{1+|x|^2}-u_\varepsilon)^{p-1}g\leq C[1+|D u_\varepsilon|^2]^{\frac{n-1}{2}},
\end{eqnarray*}
where the constant $C$ is also independent of $\varepsilon$.
The gradient estimate and interior estimates of all order can be deduced by
Lemma \ref{D-est} in Appendix. Thus, $u_\varepsilon$ must converge along some subsequence
to some $u \in C^{\infty}(U)\cap C^{0, 1}(\overline{U})$ which is a solution to the following Dirichlet problem
\begin{equation}\label{Nueq}
\left\{
\begin{aligned}
&\det (D^2u)=F(x, u, Du) \quad \mbox{in}  \quad U,
&\\
&u=0 \quad \mbox{on}  \quad \partial U,
\end{aligned}
\right.
\end{equation}
where
\begin{eqnarray*}
F(x, u, Du)=[|D u|^2
+(x \cdot Du-u)^2]^{\frac{n-q}{2}}(-u)^{p-1}g.
\end{eqnarray*}
By the argument of Theorem 1.3 in \cite{Sa14} and Theorem 1.2 in \cite{LS17}, we can derive
that $u \in C^{2, \alpha}(\overline{U})$ for $p\geq1$ and thus
$u \in C^{\infty}(\overline{U})$ which is a solution to \eqref{Nueq}.
Hence, the Dirichlet problem \eqref{Eq} admits a non-zero solution $h \in C^{\infty}(\overline{\Omega})$
for $q>p\geq 1$.

We follow the idea of the proof of Theorem 4.1 in \cite{Tso90} to show the uniqueness.
Let us suppose that the Dirichlet problem \eqref{Nueq} has two solutions $u$ and $v$ with $u-v$ being positive
somewhere in $U$. Assume that the origin is contained in $U$ and we define
\begin{eqnarray*}
u_\lambda(x)=u(\lambda^{-1}x)
\end{eqnarray*}
for $\lambda>1$ and let
\begin{eqnarray*}
\eta_\lambda(x)=\frac{-v(x)}{-u_\lambda(x)}.
\end{eqnarray*}
There exist $\varepsilon$ and $\lambda^\star$ such that
\begin{eqnarray*}
\eta_\lambda(x_\lambda)=\max_{\overline{U}}\eta_\lambda(x)\geq 1+\varepsilon
\end{eqnarray*}
for all $1<\lambda\leq \lambda^\star$. At $x_\lambda$, we have
\begin{eqnarray*}
\frac{Dv}{v}=\frac{\lambda^{-1}D u(\lambda^{-1}x)}{u_{\lambda}}, \quad
v_{ij}=\eta_\lambda(u_{\lambda})_{ij}+(\eta_\lambda)_{ij}u_{\lambda}
\end{eqnarray*}
and the matrix $(\eta_\lambda)_{ij}$ is non-positive definite. Therefore, we have at $x_\lambda$
\begin{eqnarray*}
F(x, v, Dv)&=&\mathrm{det} (D^2 v)
\\&\geq& \eta_{\lambda}^{n-1}(x)\lambda^{-2(n-1)}\mathrm{det} (D^2 u(\lambda^{-1}x))
\\&=& \eta_{\lambda}^{n-1}(x)\lambda^{-2(n-1)}F(\lambda^{-1}x, u(\lambda^{-1}x), D u(\lambda^{-1}x)).
\end{eqnarray*}
Letting $\lambda\rightarrow1$, we conclude that $(1+\varepsilon)^{p-q}\geq 1$ which yields a contradiction.
Thus, the uniqueness follows.

Moreover, we know from the equality \eqref{JJJ} that
$\mathcal{J}(h)\rightarrow+\infty$ as $\parallel h\parallel_{C^0(\Omega)}\rightarrow+\infty$ for $q>p\geq1$.
This implies that the unique solution of the Dirichlet problem \eqref{Eq} is the minimum of the functional.
\end{proof}

\subsection{Supercritical case}

We first solve the perturbed equation \eqref{1-Eq}. Let
\begin{eqnarray*}
\mathcal{I}_{\varepsilon}(h):=V_q(h)-\frac{1}{p}\int_{\Omega}\Big[(\varepsilon-h)^p-\varepsilon^p \Big]f(x)dx,
\end{eqnarray*}
which is just different from $\mathcal{J}_{\varepsilon}$ by a constant. Using Lemma \ref{Vq}, we have
\begin{equation*}
\mathcal{I}_{\varepsilon}(h)\geq a\parallel h \parallel_{C^{0}(\Omega)}^{q}
-b\parallel h \parallel_{C^{0}(\Omega)}^{p}-O(\varepsilon),
\end{equation*}
where $a, b$ are constants depending on $q, p$ and $\Omega$.
It follows that
\begin{equation*}
\mathcal{I}_{\varepsilon}(h)\geq \frac{a}{2}\frac{p-q}{p}\Big(\frac{qa}{pb}\Big)^{\frac{q}{p-q}}
\quad \mbox{for} \ \parallel h \parallel_{C^{0}(\Omega)}=\Big(\frac{qa}{pb}\Big)^{\frac{1}{p-q}}.
\end{equation*}
Set
\begin{equation*}
\sigma:=\Big(\frac{qa}{pb}\Big)^{\frac{1}{p-q}}, \quad \delta:=\frac{a}{2}\frac{p-q}{p}\Big(\frac{qa}{pb}\Big)^{\frac{q}{p-q}},
\end{equation*}
and
\begin{eqnarray}\label{CCC0}
\mathcal{C}_0:=\{h \in C^{\infty}(\overline{\Omega}): \nabla^2h+hI>0 \ \mbox{and} \ h<0 \ \mbox{in} \ \Omega,
\ h|_{\partial\Omega}=0\}.
\end{eqnarray}

By scaling of $h$, there exists $h_0, h_1 \in \mathcal{C}_0$ such that
\begin{eqnarray*}
\parallel h_0\parallel_{C^{0}(\Omega)}<\sigma<\parallel h_1\parallel_{C^{0}(\Omega)}, \quad
\mathcal{I}_\varepsilon(h_0)<\delta, \quad \mathcal{I}_\varepsilon(h_1)<\delta.
\end{eqnarray*}
Thus, the set
\begin{eqnarray*}
\mathcal{P}:=\{\gamma: [0, 1]\mapsto \mathcal{C}_0:
\parallel \gamma(0)\parallel_{C^{0}(\Omega)}<\sigma<\parallel\gamma(1)\parallel_{C^{0}(\Omega)}, \
\mathcal{I}_\varepsilon(\gamma(0))<\delta, \ \mathcal{I}_\varepsilon(\gamma(1))<\delta\}
\end{eqnarray*}
is nonempty and
\begin{eqnarray*}
c=\inf_{\gamma \in \mathcal{P}}\sup_{s \in [0, 1]}\mathcal{I}_\varepsilon(\gamma(s))\geq \delta>0.
\end{eqnarray*}
We will show that $c$ is a critical value of $\mathcal{I}_\varepsilon$ which is attained by some
$h \in \mathcal{C}_0$.

\begin{theorem}\label{D-1}
For $p>q\geq n$, the Dirichlet problem \eqref{1-Eq} admits a non-zero solution
$h \in C^{\infty}(\overline{\Omega})$ with $\mathcal{I}_\varepsilon(h)=c$.
\end{theorem}

\begin{proof}
The proof follows from a mountain-pass lemma as \cite{Tso90}.
For $0<\sigma<1$, pick a path $\gamma \in \mathcal{P}$ such that
\begin{eqnarray*}
\mathcal{I}_\varepsilon(\gamma)=\sup_{s \in [0, 1]}\mathcal{I}_\varepsilon(\gamma(s))< c+\sigma.
\end{eqnarray*}
Then, we have

\begin{lemma}
If $q\geq n$, the parabolic equation \eqref{t-Eq} with initial data
$h(x, 0, s):=\gamma(s)$ has a solution $\gamma(t, s):=h(x, t, s)$ for all $t\geq 0$.
\end{lemma}

\begin{proof}
Since the equivalence of the flow \eqref{t-Eq} and the flow \eqref{u-Eq},
it sufficient to prove the long time existence of the flow \eqref{u-Eq} for $q\geq n$.
Let $v: B_r(0)\subset \Omega\rightarrow \mathbb{R}$ solving
\begin{equation*}
\left\{
\begin{aligned}
&\det (D^2 v)=[|D v|^2
+(x \cdot Dv-v)^2]^{\frac{n-q}{2}}\varepsilon^{p-1}\inf_{B_r(0)}g \quad \mbox{in}  \quad B_r(0),
&\\
&v=0 \quad \mbox{on}  \quad \partial B_r(0).
\end{aligned}
\right.
\end{equation*}
Then, $v$ is a supersolution of \eqref{u-Eq}. Since $q\geq n$,
the comparison principle (see Theorem 14.1 in \cite{Li98}) tells us that it holds for any solution $u(x, t)$ to
the equation \eqref{u-Eq}
\begin{eqnarray*}
|u(0, t)|\geq |v(0)|.
\end{eqnarray*}
Thus, we have by the convexity of $u$
\begin{eqnarray}\label{Gra-2}
x \cdot Du-u\geq \inf_{U}[x \cdot Du-u]=-u(0, t)\geq |v(0)|>0.
\end{eqnarray}
Thus,
\begin{eqnarray*}
|G_u|=\Big|\frac{q-n}{2}\frac{2(x \cdot Du-u)}{|D u|^2+(x \cdot Du-u)^2}+\frac{1-p}
{\varepsilon\sqrt{1+|x|^2}-u}\Big|\leq C.
\end{eqnarray*}
Using the maximal principle to the evolution equation of $u_t$, we have $|u_t|\leq Ce^{Ct}$. Thus,
\begin{eqnarray*}
-C(T)\leq u(x, t)<0, \quad \forall \ (x, t) \in \overline{U} \times [0, T]
\end{eqnarray*}
and
\begin{eqnarray*}
G(x, u, Du)\leq \frac{n-1}{2}\log(1+|Du|^2)+C(T),
\quad \forall \ x \in \overline{U}.
\end{eqnarray*}
The above two inequalities imply the assumptions \eqref{LGC0} and \eqref{LGC} are satisfied.
Hence, we can use Lemma \ref{A-gra} and Lemma \ref{D2-i} to conclude that the
evolution equation \eqref{u-Eq} is uniformly parabolic on $[0, T]$.
Thus, the result of \cite{KS.IANSSM.44-1980.161} and the standard
parabolic theory show that the solution of \eqref{u-Eq} exists for all
time.
\end{proof}

Obviously, $\gamma(t, s)$ belongs to $\mathcal{P}$ for each $t\geq 0$.
Now for each $s \in [0, 1]$, we define
\begin{eqnarray*}
t^*(s):=\sup\{t\geq 0: \mathcal{I}_\varepsilon(\gamma(t, s))\geq c-\sigma\}
\end{eqnarray*}
and set $t^*(s)=0$ if $\mathcal{I}_\varepsilon(\gamma(t, s))< c-\sigma$ for all $t$.

\begin{lemma}
There exists $s_0 \in [0, 1]$ such that $t^*(s_0)=+\infty$.
\end{lemma}

\begin{proof}
We first show that $t^*$ cannot have a uniform upper bound, say $T<+\infty$.
Otherwise, there exists $T<+\infty$ such that
\begin{eqnarray*}
\gamma(T, s) \in \mathcal{P} \quad \mbox{and} \quad \sup_{s}\mathcal{I}_\varepsilon(\gamma(T, s))\leq c-\sigma.
\end{eqnarray*}
This will contradict with the definition of
$c$. Therefore, there exists a sequence $\{s_k\}$ with
\begin{eqnarray*}
\lim_{k\rightarrow +\infty}s_k=s_0 \quad \mbox{and} \quad \lim_{k\rightarrow +\infty}t^*(s_k)=+\infty.
\end{eqnarray*}

We want to prove that $t^*(s_0)=+\infty$. If this is not true, we claim $\gamma(t^*(s_0), s_0)$
is actually a solution of \eqref{1-Eq}. Indeed, if $\frac{d}{dt}\mathcal{I}_\varepsilon$ does not
vanish at $\gamma(t^*(s_0), s_0)$, we can find $t^{\prime}>t^*(s_0)$ such that
\begin{eqnarray*}
\mathcal{I}_\varepsilon(\gamma(t^{\prime}, s_0))<c-\sigma.
\end{eqnarray*}
However,
\begin{eqnarray*}
\lim_{k\rightarrow +\infty}\mathcal{I}_\varepsilon(\gamma(t^{\prime}, s_k))=\mathcal{I}_\varepsilon(\gamma(t^{\prime}, s_0))
\end{eqnarray*}
implies that for $k$ large enough
\begin{eqnarray*}
\mathcal{I}_\varepsilon(\gamma(t^{\prime}, s_k))<c-\sigma.
\end{eqnarray*}
Thus, $t^*(s_k)<t^{\prime}$ which yields a contradiction by letting $k$ go to infinity.
Thus, $\gamma(t^*(s_0), s_0)$ is actually a solution of \eqref{1-Eq}. So,
$\gamma(t^*(s_0), s_0)$ also solves \eqref{t-Eq} with itself as initial datum.
Hence, $t^*(s_0)=+\infty$ which is a contradiction. Therefore, $t^*(s_0)=+\infty$.
\end{proof}

It follows that there exists a long time solution $h(x, t):=\gamma(t, s_0)$ to the flow \eqref{t-Eq} such that
for all $t>0$
\begin{eqnarray*}
\mathcal{I}_\varepsilon(h(\cdot, t))\geq c-\sigma.
\end{eqnarray*}
Then, using the monotonicity of $ \mathcal{I}_\varepsilon(h(\cdot, t))$
along the parabolic flow \eqref{t-Eq}, for any $\sigma>0$, we can choose $T$ sufficiently large such that
\begin{eqnarray}\label{T4.10}
\int_{T}^{+\infty}\Big(-\frac{d}{d t}I_{\varepsilon}(h(\cdot, t))\Big)dt
=\int_{T}^{+\infty}\int_{\Omega}(A-B)(\log A-\log B)dxdt\leq 2\sigma,
\end{eqnarray}
where
\begin{eqnarray*}
A=[|\nabla h|^2+h^2]^{\frac{q-n}{2}}\det (\nabla^2h+hI), \quad B=(\varepsilon-h)^{p-1}f.
\end{eqnarray*}

\begin{lemma}
We have for $t\geq 0$
\begin{eqnarray}\label{C0-2}
\parallel h(\cdot, t)\parallel_{C^0(\Omega)}<C.
\end{eqnarray}
\end{lemma}

\begin{proof}
Using \eqref{T4.10} and the mean value theorem, we know that for
every interval $[k, k+1]$ with $k$ large enough, there exists $t_k \in [k, k+1]$
such that
\begin{eqnarray}\label{T3.3}
\int_{\Omega}(A(t_k)-B(t_k))(\log A(t_k)-\log B(t_k)) dx\leq 2\sigma.
\end{eqnarray}
From now on, we will fix such a $t_k$ and assume all quantities are evaluated at the chosen time
$t_k$ and suppress the dependence on $t_k$.

Let $\alpha>0$ be given by $e^{-\alpha}=1-\frac{p-q}{2p}$, and define $S\subset \Omega$ to be
the set
\begin{eqnarray*}
S:=\{x: |\log A(x)-\log B(x)|\leq \alpha\}.
\end{eqnarray*}
Then, we get from \eqref{T3.3}
\begin{eqnarray}\label{T3.4}
2\sigma&\geq&\int_{\Omega\setminus S}(A-B)(\log A-\log B)dx
\nonumber\\&\geq&\alpha\int_{\Omega\setminus S}|A-B|dx
\nonumber\\&\geq&\alpha(1-e^{-\alpha})\int_{\Omega\setminus S}B dx
\\ \nonumber&\geq&\alpha(1-e^{-\alpha})\varepsilon|\Omega\setminus S|\inf_{\Omega} f,
\end{eqnarray}
where we used $|\frac{A}{B}-1|\geq 1-e^{-\alpha}$ for all $x \in \Omega\setminus S$
to get the last  inequality.
Thus,
\begin{eqnarray*}
|\Omega\setminus S|\leq\frac{2\sigma}{\alpha(1-e^{-\alpha})\varepsilon\inf_{\Omega} f}.
\end{eqnarray*}
Choosing $\sigma$ small enough such that
\begin{eqnarray}\label{T3.5}
|\Omega\setminus S|\leq\frac{1}{2}|\Omega|.
\end{eqnarray}
Note that
\begin{eqnarray}\label{TT}
\mathcal{J}_{\varepsilon}(h)=\int_{S}\bigg[\frac{(-h)A }{q}-\frac{(\varepsilon-h)B}{p} \bigg]dx
+\int_{\Omega\setminus S}\bigg[\frac{(-h)A }{q}-\frac{(\varepsilon-h)B}{p} \bigg]dx.
\end{eqnarray}
We first estimate the first term in \eqref{TT} by using the fact $\frac{A}{B}\geq e^{-\alpha}$ for all $x \in S$
\begin{eqnarray}\label{T3.6}
&&\int_{S}\bigg[\frac{(-h)A }{q}-\frac{(\varepsilon-h)B}{p} \bigg]dx
\nonumber\\&\geq&\int_{S}\bigg[\frac{(-h)(A-B)}{q}+\frac{(p-q)(-h)B}{pq} \bigg]dx-C\varepsilon\parallel h\parallel_{C^0(\Omega)}^{p-1}
\nonumber\\&\geq&-\int_{S}\bigg[\frac{(1-e^{-\alpha})(-h)B}{q}+\frac{(p-q)(-h)B}{pq} \bigg]dx-C\varepsilon\parallel h\parallel_{C^0(\Omega)}^{p-1}
\nonumber\\&\geq&\frac{(p-q)}{2pq}\int_{S}(-h)B dx-C\varepsilon\parallel h\parallel_{C^0(\Omega)}^{p-1}
\nonumber\\&\geq&\frac{(p-q)}{2pq}\int_{S}(-h)^p f dx-C\varepsilon\parallel h\parallel_{C^0(\Omega)}^{p-1}.
\end{eqnarray}
From the fact (see Page 439 in \cite{Tso90})
\begin{eqnarray*}
\lim_{R\rightarrow +\infty}\frac{1}{|U|}\inf\bigg|\Big\{x \in U: D^2 u\geq 0, \ u(x)\leq -R/2, \ \parallel u\parallel_{C^0(U)}=R\Big\}\bigg|=1,
\end{eqnarray*}
Using the relation of $h$ and $u$ in Lemma \ref{app-lem1}, we also have
\begin{eqnarray*}
\lim_{R\rightarrow +\infty}\frac{1}{|\Omega|}\inf\bigg|\Big\{x \in \Omega: D^2 h+hI\geq 0, \ h(x)\leq -R/2, \ \parallel h\parallel_{C^0(\Omega)}=R\Big\}\bigg|=1.
\end{eqnarray*}
Then, we conclude from the above fact and \eqref{T3.5} that there exists $R_1>0$ such that
\begin{eqnarray*}
\Big|S\cap\{x \in \Omega: h(x)\leq -R_1/2\}\Big|\geq\frac{1}{4}|\Omega|
\quad \mbox{for} \quad \parallel h\parallel_{C^0(\Omega)}\geq R_1.
\end{eqnarray*}
Let $E:=\{x \in \Omega: h(x)\leq -\parallel h\parallel_{C^0(\Omega)}/2\}$. If
$\parallel h\parallel_{C^0(\Omega)}\geq R_1$, we have by \eqref{T3.6}
\begin{eqnarray}\label{T3.7}
&&\int_{S}\bigg[\frac{(-h)A }{q}-\frac{(\varepsilon-h)B}{p} \bigg]dx
\nonumber\\&\geq&\frac{p-q}{2pq}\int_{S\cap E}(-h)^p f dx-C\varepsilon\parallel h\parallel_{C^0(\Omega)}^{p-1}
\nonumber\\&\geq&\frac{(p-q)\inf_{\Omega} f}{2^{p+3}pq|\Omega|} \parallel h\parallel_{C^0(\Omega)}^{p}
-C\varepsilon\parallel h\parallel_{C^0(\Omega)}^{p-1}.
\end{eqnarray}
For the second term in \eqref{TT}, we obtain from \eqref{T3.4}
\begin{eqnarray}\label{T3.66}
&&\int_{\Omega\setminus S}\bigg[\frac{(-h)A }{q}-\frac{(\varepsilon-h)B}{p} \bigg]dx
\nonumber\\&\geq&\frac{1}{q}\int_{\Omega\setminus S}(-h)A dx-\frac{1}{p}\int_{\Omega\setminus S}(-h)B dx-C\varepsilon\parallel h\parallel_{C^0(\Omega)}^{p-1}
\nonumber\\&\geq&-\frac{1}{p}\parallel h\parallel_{C^0(\Omega)}\int_{\Omega\setminus S}B dx-C\varepsilon\parallel h\parallel_{C^0(\Omega)}^{p-1}
\nonumber\\&\geq&-\frac{2\sigma}{p\alpha(1-e^{-\alpha})}\parallel h\parallel_{C^0(\Omega)}-C\varepsilon\parallel h\parallel_{C^0(\Omega)}^{p-1}.
\end{eqnarray}
Combining \eqref{T3.7} and \eqref{T3.66}, it yields that
\begin{eqnarray*}
c+\sigma\geq\mathcal{I}_{\varepsilon}(h)\geq
\frac{p-q}{2^{p+3}pq|\Omega|}\parallel h\parallel_{C^0(\Omega)}^{p}-\frac{\sigma}{p\alpha(1-e^{-\alpha})}\parallel h\parallel_{C^0(\Omega)}-C\varepsilon\parallel h\parallel_{C^0(\Omega)}^{p-1}.
\end{eqnarray*}
Thus, $\parallel h(\cdot, t_k)\parallel_{C^0(\Omega)}\leq C$. We know from \eqref{T4.10}
\begin{eqnarray*}
\int_{T}^{+\infty}\int_{\Omega}B(e^{h_t}-1)h_t dxdt\leq 2\sigma.
\end{eqnarray*}
Using the simple facts $-\frac{x}{2}\leq(e^x-1)x$ for $x\leq -1$ and $\frac{1}{3}x^2\leq(e^x-1)x$ for $-1\leq x\leq 0$, we have
\begin{eqnarray*}
&&\frac{1}{p}\int_{\Omega}[\varepsilon-h(t)]^pf dx
\\&=&\frac{1}{p}\int_{\Omega}[\varepsilon-h(t_k)]^pf dx-
\int_{t_k}^{t}\int_{\Omega}B h_t dxdt
\\&\leq&C-\int_{t_k}^{t}\int_{\{h_t\leq-1\}}B h_t dx dt-
\int_{t_k}^{t}\int_{\{0>h_t>-1\}}B h_t dxdt
\\&\leq&C+2\int_{t_k}^{t}\int_{\{h_t\leq-1\}}B h_t(e^{h_t}-1)dx dt-
\int_{t_k}^{t}\int_{\{0>h_t>-1\}}B h_t dx dt
\\&\leq&C+4\sigma+
\bigg(\int_{t_k}^{t}\int_{\{0>h_t>-1\}}B dx dt\bigg)^{\frac{1}{2}}
\bigg(\int_{t_k}^{t}\int_{\{0>h_t>-1\}}Bh_{t}^{2} dx dt\bigg)^{\frac{1}{2}}
\\&\leq&C+4\sigma+
\bigg(\int_{t_k}^{t}\int_{\Omega}B dx dt\bigg)^{\frac{1}{2}}
\bigg(3\int_{t_k}^{t}\int_{\{0>h_t>-1\}}B(e^{h_t}-1)h_{t} dx dt\bigg)^{\frac{1}{2}}
\\&\leq&C+4\sigma+
\sqrt{6\sigma}\bigg(\int_{t_k}^{t}\int_{\Omega}B dx dt\bigg)^{\frac{1}{2}}.
\end{eqnarray*}
Thus,
\begin{eqnarray*}
\frac{\varepsilon}{p}\int_{\Omega}B(t) dx
\leq C+4\sigma+
\sqrt{6\sigma}\bigg(\int_{t_k}^{t}\int_{\Omega}B dx dt\bigg)^{\frac{1}{2}}.
\end{eqnarray*}
Set
\begin{eqnarray*}
M_k=\max\bigg\{\int_{\Omega}B(t) dx: t \in [t_k, t_{k+1}]\bigg\}.
\end{eqnarray*}
Then,
\begin{eqnarray*}
\frac{\varepsilon}{p}M_k
\leq C+4\sigma+
\sqrt{6\sigma}M_{k}^{\frac{1}{2}}.
\end{eqnarray*}
This implies an upper bound on $M_k$ if we choose $\sigma$ small enough. Then, a uniform upper
bound on $\parallel h(\cdot, t)\parallel_{C^0(\Omega)}$ follows from the convexity.
\end{proof}

From the gradient estimate \eqref{Gra-2} and the $C^0$ estimate \eqref{C0-2}, we obtain that the right part of the equation in \eqref{u-Eq} satisfies
\begin{eqnarray*}
G(x, u, Du)\leq \frac{n-1}{2}\log(1+|Du|^2)+C,
\quad \forall \ x \in \overline{U}.
\end{eqnarray*}
Then, we can use Lemma \ref{A-gra} and Lemma \ref{D2-i} in Appendix to conclude that the
evolution equation \eqref{PMA} is uniformly parabolic on any finite time
interval. Thus, the result of \cite{KS.IANSSM.44-1980.161} and the standard
parabolic theory show that the solution of \eqref{PMA} exists for all
time. By these estimates again, a subsequence of $h(\cdot, t)$ converges to
a function $h_{\sigma}$ satisfies $c\leq\mathcal{I}_{\varepsilon}(h_{\sigma})\leq c+\sigma$.
Let $\sigma\rightarrow 0$, $h_\sigma$ converges to
a function $h$ that satisfies $\mathcal{I}_{\varepsilon}(h)=c$.
\end{proof}

\begin{theorem}
The Dirichlet problem \eqref{Eq} admits a non-zero solution $h \in C^{\infty}(\overline{\Omega})$ for $p>q\geq n$.
\end{theorem}

\begin{proof}
Using Theorem \ref{D-1}, there exists a solution $h_{\varepsilon}$ with
$\mathcal{I}_{\varepsilon}(h_{\varepsilon})=c$ solving \eqref{1-Eq}.
Hence,
\begin{eqnarray*}
\mathcal{I}_{\varepsilon}(h_{\varepsilon})=\frac{1}{q}\int_{\Omega}(-h_{\varepsilon})(\varepsilon-h_{\varepsilon})^{p-1}f dx
-\frac{1}{p}\int_{\Omega}[(\varepsilon-h_{\varepsilon})^p-\varepsilon^p] f dx.
\end{eqnarray*}
Since $\delta\leq \mathcal{I}_{\varepsilon}(h_{\varepsilon})\leq C$, we have
\begin{eqnarray*}
C^{-1} \leq \parallel h_\varepsilon\parallel_{C^0(\Omega)}\leq C,
\end{eqnarray*}
where the constant $C$ is independent of $\varepsilon$. Following the same argument in Theorem \ref{D},
$h_\varepsilon$ must converge along some subsequence
to a non-zero function $h \in C^{2, \alpha}(\overline{\Omega})$ and thus $h \in C^{\infty}(\overline{\Omega})$ which is a solution to the Dirichlet problem \eqref{Eq}.
\end{proof}

\subsection{Critical case}

In this subsection, we use the variational method to study the solvability of the equation \eqref{Eq} in the case $p=q\geq 1$.

\begin{theorem}\label{q-critical-exist}
Let $p\geq 1$, there exists a constant $\lambda>0$ such that the Dirichlet problem
\begin{equation}\label{Eq-critical-21}
\left\{
\begin{aligned}
&\det (\nabla^2h+hI)= \lambda f(-h)^{p-1} [|\nabla h|^2+h^2]^{\frac{n-p}{2}}\quad  \mbox{in} \quad \Omega,&\\
&h=0 \quad  \mbox{on} \quad \partial \Omega
\end{aligned}
\right.
\end{equation}
admits a non-zero solution $h\in C^{\infty} (\overline{\Omega})$.
Moreover, if $(\lambda_1, h_1)$ is another pair of such solution, then $\lambda=\lambda_1$ and  there exists a constant $c$ such that $h_1=ch$.
\end{theorem}

\begin{proof}

We divide the proof into three steps.

\textbf{Step 1:} Existence of solutions to  the equation \eqref{Eq-critical-21} in the case $p>1$.

Let $s\in (1, p)$, we consider a family of equations depending on $s$:
\begin{equation}\label{Eq-critical-2121}
\left\{
\begin{aligned}
&\det (\nabla^2h+hI)= \lambda  f(-h)^{s-1} [|\nabla h|^2+h^2]^{\frac{n-p}{2}}\quad  \mbox{in} \quad \Omega,&\\
&h=0\quad  \mbox{on} \quad \partial \Omega,
\end{aligned}
\right.
\end{equation}
where $\lambda$ is an invariant defined by
\begin{eqnarray*}
\lambda (\Omega):=\inf_{h\in \mathcal{C}_0} \left\{\frac{pV_p(h)}{\int_\Omega (-h)^p f(x)dx} \right\}
\end{eqnarray*}
and see \eqref{CCC0} for the definition of $\mathcal{C}_0$. Clearly, $\lambda (\Omega)>0$ by Lemma \ref{Vq}.

Note that \eqref{Eq-critical-2121} is the Euler-Lagrange equation of the functional
$$\mathcal{J}_s(h):=V_p(h)-  \frac{\lambda}{s}\int_{\Omega} (-h)^{s} f(x) dx.$$
By Theorem \ref{D}, the equation \eqref{Eq-critical-2121} admits a unique solution $h_{s}$ which minimizes the functional $\mathcal{J}_s(h)$.
Now we consider the sequence of rescaled solutions
$$v_s:=\frac{h_s}{\|h_s\|_{C^{0}(\Omega)}}.$$
Then,
$$\|h_s\|_{C^{0}(\Omega)}^{p-s}=\frac{\lambda \int_{\Omega} (-v_s)^sf(x)dx}{pV_p(v_s)}.$$
By H\"older inequality and the definition of $\lambda$,
\begin{eqnarray*}
\|h_s\|_{C^{0}(\Omega)}^{p-s} &\leq& \frac{\int_{\Omega} (-v_s)^sf(x)dx}{\int_{\Omega} (-v_s)^pf(x)dx}
\nonumber\\&\leq &  \frac{\left( \int_{\Omega} (-v_s)^pf(x)dx\right)^{\frac{s}{p}}\left(\int_{\Omega} f dx\right)^{\frac{p-s}{p}}}{\int_{\Omega} (-v_s)^pf(x)dx}\\
&=&  \frac{\left(\int_{\Omega} f dx\right)^{\frac{p-s}{p}}}{\left( \int_{\Omega} (-v_s)^pf(x)dx\right)^{\frac{p-s}{p}}}.
\end{eqnarray*}
Thus
$$\|h_s\|_{C^{0}(\Omega)} \leq A^{-\frac{1}{p}} \left(\int_{\Omega} f dx\right)^{\frac{1}{p}},$$
where $A=\inf\{  \int_{\Omega} (-h)^pf(x)dx : h\in \mathcal{C}_0, \|h\|_{C^{0}(\Omega)}=1\}$ is a constant depending on $p, f$ and $\Omega$.

According to the definition of $\lambda$, there exists a function $\tilde{h}\in \mathcal{C}_0$ such that $\|\tilde{h}\|_{C^{0}(\Omega)}=1$ and
$$\left(\frac{\lambda \int_{\Omega} (-\tilde{h})^pf(x)dx}{pV_p(\tilde{h})}\right)^{\frac{1}{p-s}} \geq \frac{1}{2}.$$
Then
$$a:=\left(\frac{\lambda \int_{\Omega} (-\tilde{h})^sf(x)dx}{pV_p(\tilde{h})} \right)^{\frac{1}{p-s}}\geq \frac{1}{2}.$$
On one hand, we know that from the equation \eqref{Eq-critical-2121}
$$\mathcal{J}_s(h_s)=\frac{s-p}{sp} \lambda \int_{\Omega}(-h_s)^s f(x)dx.$$
On the other hand,
\begin{eqnarray*}
\mathcal{J}_s(h_s) &\leq&\mathcal{J}_s(a\tilde{h})
\nonumber\\&= & \frac{a^{p}}{p} pV_p(\tilde{h})-  \frac{\lambda a^s}{s}\int_{\Omega} (-\tilde{h})^{s} f(x) dx\\
&=& \frac{s-p}{sp} a^s \lambda \int_{\Omega} (-\tilde{h})^{s} f(x) dx\nonumber\\
&\leq& \frac{s-p}{sp} 2^{-s} \lambda \int_{\Omega} (-\tilde{h})^{p} f(x) dx<0.
\end{eqnarray*}
It implies that
$$ \|h_s\|_{C^{0}(\Omega)}^s \int_{\Omega}  f(x)dx  \geq  \int_{\Omega} (-h_s)^s f(x)dx \geq 2^{-s} A.$$
Thus,
$$  \|h_s\|_{C^{0}(\Omega)}\geq 2^{-1} \left( \frac{A}{\int_{\Omega}  f(x)dx} \right)^{\frac{1}{s}} \geq 2^{-1} \left( \frac{A}{\int_{\Omega}  f(x)dx} \right).$$
Therefore, $h_s$ has a uniform upper bound and a uniform positive lower bound.
Then, by the argument in the proof of Theorem \ref{D},
$h_s$ converges to a smooth solution of the equation \eqref{Eq-critical-21} in the case $p>1$.

\textbf{Step 2:}  Existence of solutions to the equation \eqref{Eq-critical-21} in the case $p=1$.

Let $p_{\epsilon}=1+\epsilon$. From Step 1, we know that there exist constants $\lambda_{\epsilon}$ and a non-zero
function $h_{\epsilon}$ satisfying the equation
\begin{equation*}
\left\{
\begin{aligned}
&\det (\nabla^2h+hI)= \lambda_{\epsilon}  f(-h)^{p_{\epsilon}-1} [|\nabla h|^2+h^2]^{\frac{n-p_{\epsilon}}{2}}\quad  \mbox{in} \quad \Omega ,&\\
&h=0 \quad  \mbox{on} \quad \partial \Omega.
\end{aligned}
\right.
\end{equation*}
There is no loss of generality in assuming $\|h_{\epsilon}\|_{C^0}=1$.
It is clear that $\lambda_{\epsilon} \leq C$ uniformly in $\epsilon$.
Then, using the same argument in the proof of Step 1, we conclude that there is a nonzero solution $h\in C^{\infty}(\bar{\Omega})$ to the  equation \eqref{Eq-critical-21} in the case $p=1$.

\textbf{Step 3:} Uniqueness of solutions to the equation \eqref{Eq-critical-21}.

Suppose that $(\lambda_1, h_1)$ and $(\lambda_2, h_2)$ are two pairs which solve the equation \eqref{Eq-critical-21}.
Then, using Lemma \ref{app-lem1}, the transform of $h_i$ $(i=1,2)$, denoted by $u_i$, satisfies the following equation
\begin{equation}\label{Eq-critical-22222}
\left\{
\begin{aligned}
&\det (D^2u_i)= \lambda_i g(x)(-u_i)^{p-1} [|D u_i|^2+(x\cdot Du_i-u_i)^2]^{\frac{n-p}{2}}\quad  \mbox{in} \quad U ,&\\
&u_i=0 \quad  \mbox{on}  \quad \partial U.
\end{aligned}
\right.
\end{equation}
Without loss of generality, we can assume that $\lambda_1\leq \lambda_2$.
Since $u_i$ is convex in $U$,
$$\frac{\partial u_i}{\partial \nu}>0\quad \mbox{on} \quad \partial U,$$
where $\nu$ is the unit outward normal to $\partial U$. Thus, for some $t>0$ small, we have
\begin{eqnarray*}
0\leq t(-u_2)\leq -u_1 \quad \mbox{on} \quad \overline{U}.
\end{eqnarray*}
Thus,
\begin{eqnarray*}
t_0:= \sup \{t>0 : t(-u_2)\leq -u_1~\mbox{on}~\overline{U}\}>0.
\end{eqnarray*}
Note that any scaling $(\lambda_2, tu_2)$ also solves equation \eqref{Eq-critical-22222} for any $t>0$,
therefore we can replace $u_2$ by its scalings $t_0 u_2$.
We have  $u_2- u_1\geq 0$  in $U$.
We can divide the proof into two cases.

\textbf{Case 1:} $u_2\equiv u_1$ in $U$, it is easy to check that $\lambda_1= \lambda_2$.

\textbf{Case 2:} there exists a point $x_0\in  U$ such that  $u_2- u_1 > 0$ at $x_0$. We can apply similar method in the proof of uniqueness in Theorem \ref{D} to show that $\lambda_1\geq \lambda_2$. Then our theorem is now proved.
\end{proof}

\section{The Dirichlet problem in  the case $p<1$}

In this section, we establish the existence and optimal global H\"older regularity
for solutions to the Dirichlet problem \eqref{Eq} in the case $p<1$.
Using Lemma \ref{app-lem1}, we only need to consider the equation \eqref{Eq-transfer-1}.

First, we introduce the following comparison principle.

\begin{lemma}\label{app-comparison-1}
Let $U\subset \mathbb{R}^{n-1}$ be a bounded convex domain, $p<1$ and $q\geq n$. Assume that
$u, v\in C^2(U)\cap C^0(\overline{U})$ are convex functions such that
$$ 0 \geq u \geq v \quad \mbox{on} \quad \partial U,$$
and
$$\det (D^2 u)\leq g(x)(-u)^{p-1}\left[|D u|^2+(x\cdot Du-u)^2\right]^{\frac{n-q}{2}}\quad \mbox{in} \quad U,$$
$$\det (D^2 v)\geq g(x)(-v)^{p-1}\left[|D v|^2+(x\cdot Dv-v)^2\right]^{\frac{n-q}{2}}\quad \mbox{in} \quad U.$$
Then $u\geq v$ in $U$.
\end{lemma}
\begin{proof}
Assume $u-v$ attains  its minimum value at $x_0\in U$ with $u(x_0)<v(x_0)<0$, then
$$Du(x_0)=Dv(x_0), \quad D^2u(x_0) \geq D^2 v(x_0),$$
which implies at $x_0$
$$\left[|D v|^2+(x\cdot Dv-v)^2\right]^{\frac{n-q}{2}}>\left[|D u|^2+(x\cdot Du-u)^2\right]^{\frac{n-q}{2}}$$
and
$$(-v)^{p-1}\left[|D v|^2+(x\cdot Dv-v)^2\right]^{\frac{n-q}{2}} \leq (-u)^{p-1}\left[|D u|^2+(x\cdot Du-u)^2\right]^{\frac{n-q}{2}}.$$
Thus we easily find $(-v)^{p-1}(x_0)\leq (-u)^{p-1}(x_0)$, which contradicts $u(x_0)<v(x_0)$ if $p<1$.
\end{proof}

\begin{lemma}\label{app-lowest-1}
Let $U\subset \mathbb{R}^{n-1}$ be a bounded convex domain, $\epsilon>0$, $p<1$ and $q\geq n$.
Assume that $u \in C^2(U)\cap C^0(\overline{U})$ is a convex solution to the equation
\begin{equation*}
\left\{
\begin{aligned}
&\det (D^2 u)=g(x) (\epsilon-u)^{p-1}\left[|D u|^2+(x\cdot Du-u)^2\right]^{\frac{n-q}{2}}\quad \mbox{in} \quad U,&\\
&u= 0 \quad  \mbox{on} \quad  \partial U.
\end{aligned}
\right.
\end{equation*}
There exists a constant $\epsilon_0(n, p, q, g)$ such that
$$\|u\|_{C^{0}(\overline{U})}\geq c(n, p, q, \inf g)\left(|U|^{\ast}\right)^{\frac{1}{q-p}}$$
if $\epsilon< \epsilon_0(n, p, q, g)$,
where $|U|^{\ast}:= \min \{|U|^{2}, |U|^{\frac{n+q-2}{n-1}} \}$.
\end{lemma}

\begin{proof}
The proof is similar to that of  Lemma 2.3 in \cite{LNQ-1}. There is no loss of generality in assuming $U$ is normalized, i.e., there exists a constant $R$ such that
$$B_R\subset U \subset B_{(n-1)R}.$$
Let $s=\|u\|_{C^{0}(U)}$ and $v=\frac{u}{s}$, then $v\in C^2(U)\cap C^0(\overline{U})$ is a convex solution to
\begin{equation*}
\left\{
\begin{aligned}
&\det (D^2 v)=s^{1-q}g(x) (\epsilon-sv)^{p-1}\left[|D v|^2+(x\cdot Dv-v)^2\right]^{\frac{n-q}{2}}\quad  \mbox{in} \quad U,&\\
&v=0 \quad  \mbox{on} \quad \partial U
\end{aligned}
\right.
\end{equation*}
with $\|v\|_{C^{0}(\overline{U})}=1$. It follows that
 \begin{eqnarray*}
\frac{g(x)}{(s+\epsilon)^{1-p}} &\leq&\det (D^2 v) s^{q-1} \left[|D v|^2+(x\cdot Dv-v)^2\right]^{\frac{q-n}{2}} \nonumber\\
&\leq&s^{q-1}  \det (D^2 v) \left[\left(1+R^2/2\right)|D v|^2+2\right]^{\frac{q-n}{2}}
\end{eqnarray*}
for any $x\in B_{R/2}$. Integrating both sides over $ B_{R/2}$ and  using area formula, we have
 \begin{eqnarray}\label{app-lowest-for-1}
\frac{\int_{B_{R/2}} g(x)dx }{(s+\epsilon)^{1-p}} &\leq&
s^{q-1} \int_{B_{R/2}} \det (D^2 v) \left[\left(1+R^2/2\right)|D v|^2+2\right]^{\frac{q-n}{2}}  dx \nonumber\\
&=& s^{q-1} \int_{D v( B_{R/2})} \left[\left(1+R^2/2\right)|y|^2+2\right]^{\frac{q-n}{2}}  dy.
\end{eqnarray}
Note that $v\in C^0(\overline{U})$ is convex with $v=0$ on $\partial U$, it is easy to see that
\begin{eqnarray}\label{app-lowest-for-2}
|Dv(x)| \leq \frac{|v(x)|}{\mbox{dist}(x, \partial U)}\leq \frac{2}{R}
\end{eqnarray}
for any $x\in  B_{R/2}$. Substituting \eqref{app-lowest-for-2} into \eqref{app-lowest-for-1} yields
 \begin{eqnarray}\label{app-lowest-for-3}
  s^{q-1} (s+\epsilon)^{1-p} \geq \frac{\int_{B_{R/2}} g(x)dx}{(4+\frac{4}{R^2})^{\frac{q-n}{2}} (\frac{2}{R})^{n-1}|B_1|} \geq c_1(n, p, q, \inf g) |U|^{\ast},
\end{eqnarray}
which implies that
\begin{eqnarray}\label{app-lowest-for-4}
s^{q-p}2^{1-p}\max\{1, \epsilon_0/s\}^{1-p} \geq c_1(n, p, q, \inf g) |U|^{\ast}.
\end{eqnarray}
By choosing $\epsilon_0=\left(\frac{c_1(n, p, q, \inf g) |U|^{\ast}}{2^{2-p}}\right)^{\frac{1}{q-p}}$, we obtain the conclusion of the lemma 3.1.
\end{proof}

Now, we construct supersolutions to the  equation \eqref{Eq-transfer-1} with optimal global H\"older regularity, which  is similar to \cite{LNQ-1, LNQ-2}.
\begin{lemma}\label{app-sub-1}
Let  $p<1$ and $q\geq n\geq 3$, $U\subset \mathbb{R}^{n-1}$ be a bounded convex domain with $0 \in \partial U$ and $U\subset \{x=(x^\prime, x_{n-1}) \subset \mathbb{R}^{n-1}: x_{n-1} >0 \}$. Then there exists a constant $C=C(n, p, q, \mbox{diam}(U), \sup g)$ such that the following function
\begin{equation}\label{Eq-app-11}
v_a(x)=x_{n-1}^a (|x^\prime|^2-C)
\end{equation}
is smooth, convex and satisfies
\begin{equation*}
\left\{
\begin{aligned}
&\det (D^2 v_a) (-v_a)^{1-p}\left[|D v_a|^2+(x\cdot Dv_a-v_a)^2\right]^{\frac{q-n}{2}}\geq g(x)\quad  \mbox{in} \quad U,&\\
&v_a\leq 0 \quad  \mbox{on} \quad \partial U.
\end{aligned}
\right.
\end{equation*}
Here $a=\frac{q-n+2}{q-p}\in(0,1)$.
\end{lemma}
\begin{proof}
For $x=(x^\prime, x_{n-1})$, we denote $r=|x^\prime|$, then $v_a=x_{n-1}^a (r^2-C)$ and
\begin{eqnarray*}
(v_a)_r&=&2rx_{n-1}^a,\\
(v_a)_{rr}&=&2x_{n-1}^a,\\
(v_a)_{x_{n-1}}&=&ax_{n-1}^{a-1} (r^2-C),\\
(v_a)_{x_{n-1}x_{n-1}}&=&a(a-1)x_{n-1}^{a-2} (r^2-C),\\
(v_a)_{x_{n-1}r}&=&2rax_{n-1}^{a-1}.
\end{eqnarray*}
In suitable coordinate systems, such as cylindrical in $x^\prime$, the Hessian of $v_a$ has the following form
$$
D^2v_a
= \left(
\begin{array}{ccccc}
                                 \frac{(v_a)_r}{r}& 0& \cdots& 0& 0\\
                                  0& \frac{(v_a)_r}{r}& \cdots& 0& 0\\
                                  \vdots& \vdots& \ddots& \vdots& \vdots\\
                                   0 & 0 & \cdots& (v_a)_{rr}& (v_a)_{x_{n-1}r}  \\
                                   0& 0& \cdots& (v_a)_{x_{n-1}r}& (v_a)_{x_{n-1}x_{n-1}}\\
                           \end{array}\right).
$$
We have
 \begin{eqnarray*}
\det(D^2v_a) &=& \left(\frac{(v_a)_r}{r}\right)^{n-3} \left((v_a)_{x_{n-1}x_{n-1}}(v_a)_{rr}-(v_a)_{x_{n-1}r}^2\right)
\nonumber\\&=& 2^{n-2} x_{n-1}^{an-a-2} \left((a-a^2)C-(a+a^2)r^2 \right).
\end{eqnarray*}
It follows that $v_a$ is smooth and convex in $U$ provided by $C>>\mbox{diam}^2 (U)$. Then
 \begin{eqnarray*}
&&|D v_a|^2 +(x\cdot v_a-v_a)^2\\
 &=&((v_a)_{x_{n-1}}^2+(v_a)_r^2) + (x_{n-1} (v_a)_{x_{n-1}}+r (v_a)_r-v_a)^2\nonumber\\
&=& x_{n-1}^{2a-2} \left(a^2(C-r^2)^2+4r^2x_{n-1}^2\right)+x_{n-1}^{2a}[(1+a)r^2+(1-a)C]^2\nonumber\\
&\geq&x_{n-1}^{2a-2} a^2(C-r^2)^2.
\end{eqnarray*}
Therefore
\begin{eqnarray*}
&&\det (D^2 v_a) (-v_a)^{1-p}\left[|D v_a|^2+(x\cdot Dv_a-v_a)^2\right]^{\frac{q-n}{2}}\\
&\geq& 2^{n-2}  x_{n-1}^{a(q-p)-q+n-2}  \left((a-a^2)C-(a+a^2)r^2 \right)(C-r^2)^{q-n+1-p}a^{q-n} \nonumber\\
&\geq& 2^{n-2} (a+a^2) a^{q-n} \left(\frac{a-a^2}{a+a^2}C-r^2 \right)^{q-n+2-p} \nonumber\\
&\geq& g(x)
\end{eqnarray*}
if we choose $a=\frac{q-n+2}{q-p}$ and $C=C(n, p, q, \sup g) (1+\mbox{diam}^2(U))$ so large.
\end{proof}

\begin{lemma}\label{app-sup-1}
Let  $q\geq n\geq 3$ and $p<1$. For any  $a\in [\frac{q-n+2}{q-p},1)$, we denote by
$$b=\frac{q-1}{q-p}, \quad s=\frac{b}{1-a}.$$
Let $U=\{(x^\prime, x_{n-1}) \subset \mathbb{R}^{n-1}: |x^\prime|<1, 0<x_{n-1}<(1-|x^\prime|^2)^{s}\}$. Then there exists a constant $C=C(n, p, q, \inf g)$ such that the following function
\begin{equation}\label{Eq-app-11111}
w(x)=Cx_{n-1}-Cx_{n-1}^a (1-|x^\prime|)^{b}
\end{equation}
is smooth and satisfies
\begin{equation*}
\left\{
\begin{aligned}
&\det (D^2 w) (-w)^{1-p}\left[|D w|^2+(x\cdot Dw-w)^2\right]^{\frac{q-n}{2}}\leq g(x)\quad  \mbox{in} \quad U,&\\
&w=0 \quad  \mbox{on} \quad \partial U.
\end{aligned}
\right.
\end{equation*}
\end{lemma}
\begin{proof}
Denote $r=|x^\prime|$. As in the proof of Lemma \ref{app-sub-1}, we know that $w=Cx_{n-1}-Cx_{n-1}^a(1-r^2)^{b}$ and
 \begin{eqnarray*}
w_r &=& 2Cbx_{n-1}^a(1-r^2)^{b-1}r,\\
w_{rr}&=&2Cbx_{n-1}^a(1-r^2)^{b-2}[1-(2b-1)r^2],\\
w_{x_{n-1}}&=&C-Cax_{n-1}^{a-1}(1-r^2)^{b},\\
w_{x_{n-1}r}&=&2Cabx_{n-1}^{a-1}(1-r^2)^{b-1}r,\\
w_{x_{n-1}x_{n-1}}&=&-Ca(a-1)x_{n-1}^{a-2}(1-r^2)^{b}.
\end{eqnarray*}
Thus, we have
 \begin{eqnarray*}
\det(D^2w) &=& \left(\frac{w_r}{r}\right)^{n-3} \left(w_{x_{n-1}x_{n-1}}w_{rr}-w_{x_{n-1}r}^2\right)
\nonumber\\&=& 2^{n-2} C^{n-1} a b^{n-2}x_{n-1}^{an-a-2} (1-r^2)^{(b-1)(n-1)}\left(1-a+(1-2b-a)r^2\right)\\
&\leq& 2^{n-2} C^{n-1} a b^{n-2}x_{n-1}^{an-a-2} (1-r^2)^{(b-1)(n-1)} (1-a+|1-2b-a|).
\end{eqnarray*}
Hence,
 \begin{eqnarray*}
&&|D w|^2 +(x\cdot w-w)^2 \\
&=&4C^2b^2x_{n-1}^{2a} (1-r^2)^{2b-2} r^2+ C^2(1-a x_{n-1}^{a-1}(1-r^2)^b)^2\\
&&+(x_{n-1} w_{x_{n-1}}+r w_r-w)^2\nonumber\\
&\leq& 4C^2b^2x_{n-1}^{2a} (1-r^2)^{2b-2}+ 2C^2+2C^2a^2x_{n-1}^{2a-2}(1-r^2)^{2b}\\
&& + C^2x_{n-1}^{2a} (1-r^2)^{2b-2}\left((1-a)(1-r^2) +2br^2\right)^2\\
 &\leq& 2C^2+ C^2x_{n-1}^{2a-2} (1-r^2)^{2b-2} \left( 12b^2+4a^2+2-4a\right)\\
 &\leq& C^2x_{n-1}^{2a-2} (1-r^2)^{2b-2} \left( 12b^2+4a^2+4-4a\right).
\end{eqnarray*}
Therefore
\begin{eqnarray*}
&&\det (D^2 w) (-w)^{1-p}\left[|D w|^2+(x\cdot Dw-w)^2\right]^{\frac{q-n}{2}}\\
&\leq& 2^{n-2} C^{q-p} ab^{n-2} x_{n-1}^{a(q-p)-q+n-2} (1-r^2)^{b(q-p)-(q-1)}(12b^2+4a^2+4-4a)^{\frac{q-n}{2}}\\
&&\cdot(1-a+|1-2b-a|)\\
&\leq& g(x)
\end{eqnarray*}
if we choose a suitable $C=C(n,p,q, \inf g)$. Thus we complete the proof.
\end{proof}

\begin{theorem}\label{exist-results-p2}
Let $U$ be a bounded, open and convex domain in $\mathbb{R}^{n-1}$, $p<1$ and $q\geq n \geq 3$. There exist a unique nontrivial convex solution
$u\in C^{\infty}(U) \cap C^0(\overline{U})$ to the equation \eqref{Eq-transfer-1} with the estimate
\begin{equation}\label{esti-p1}
|u(x)| \leq C(n, p, q, \mbox{diam}(U), \sup g) \left(\mbox{dist}(x, \partial U)\right)^{\frac{q-n+2}{q-p}} \quad~\mbox{for any}~x\in U.
\end{equation}
Moreover, the exponent $\frac{q-n+2}{q-p}$ is optimal, i.e., for any  $a\in (\frac{q-n+2}{q-p},1)$, there exists a bounded convex domain $U\subset \mathbb{R}^{n-1}$ such that the solution of the equation \eqref{Eq-transfer-1} satisfies  $u\notin C^{a}(\overline{U})$.
\end{theorem}
\begin{proof}
We divide our proof into three steps.

\textbf{Step 1:} We show the estimate \eqref{esti-p1} holds.

Let $u\in C^{\infty}(U) \cap C^0(\overline{U})$ be the unique nontrivial convex solution to the  equation \eqref{Eq-transfer-1}, $z$ be an arbitrary point in $U$, $z_0$  be a point in $\partial U$ such that $|z-z_0|=\mbox{dist}(z, \partial U)$. Suppose that the supporting hyperplane $l_{z_0}:=\{x\in \mathbb{R}^{n-1}\mid \mathbf{n}\cdot (x-z_0)=0\}$ to $\partial U$ at $z_0$, where $\mathbf{n}$ is the inner normal unit vector to $\partial U$ at $z_0$. Then
$$U\subset \{x\in \mathbb{R}^{n-1}\mid \mathbf{n}\cdot (x-z_0)\geq 0\}.$$

Define a function
$$v(x)= [\mathbf{n}\cdot (x-z_0)]^a (|x-\mathbf{n}\cdot x|^2-C),$$
where $a=\frac{q-n+2}{q-p}$, $C$ is a large constant to be determined later.
By translation and rotation of coordinates, we can assume that  $\mathbf{n}=(0,\cdots, 0, 1)$, $z_0=0$, then $v=x_{n-1}^a (|x^\prime|^2-C)$. According to Lemma \ref{app-sub-1}, we can choose a suitable constant $C$ such that $v$ is a subsolution to equation  \eqref{Eq-transfer-1}.  Using Lemma \ref{app-comparison-1}, we have
\begin{equation}\label{app-the-for-1}
|u(z)| \leq |v(z)| \leq C|z-z_0|^a =C\left(\mbox{dist}(z, \partial U) \right)^a
\end{equation}
 By the convexity of $u$, we easily obtain $u\in C^{\frac{q-n+2}{q-p}} (\overline{U})$.

\textbf{Step 2:} We prove the existence and uniqueness of solutions to the equation \eqref{Eq-transfer-1}.

Let $U_{\epsilon}$ be a sequence of open, bounded, smooth and strictly convex domains in $\mathbb{R}^{n-1}$ such that $U_{\epsilon} \rightarrow U$ in the Hausdorff distance. Consider the following Monge-Amp\`ere equation
\begin{equation}\label{exist-epsilon-1}
\left\{
\begin{aligned}
&\det (D^2 u_{\epsilon})=g(x) (\epsilon-u_{\epsilon})^{p-1}\left[|D u_{\epsilon}|^2+(x\cdot Du_{\epsilon}-u_{\epsilon})^2\right]^{\frac{n-q}{2}}\quad  \mbox{in} \quad U_{\epsilon} ,&\\
&u_{\epsilon}=0 \quad  \mbox{on} \quad \partial U_{\epsilon} ,
\end{aligned}
\right.
\end{equation}
where $\epsilon<\epsilon_0$, which is given in Lemma \ref{app-lowest-1}. From Theorem 7.1 in \cite{CNS1},
there exists a unique convex solution $u_{\epsilon}\in C^{\infty}(\overline{U_{\epsilon} })$ to the equation \eqref{exist-epsilon-1}. Lemma \ref{app-lowest-1} implies that there exists a constant $c(p, q, n, \inf g)$ such that
$$\|u_{\epsilon}\|_{C^{0}(\overline{U})} \geq c(p, q, n, \inf g) \left(|U_{\epsilon}|^{\ast}\right)^{\frac{1}{q-p}}.$$
We now apply the same argument in Step 1 to obtain that
$$|u_{\epsilon}|(x)\leq C(n, p, q, \mbox{diam}(U), \sup g) \left(\mbox{dist}(x, \partial U_{\epsilon} )\right)^{\frac{q-n+2}{q-p}} \quad~\mbox{for any}~x\in U_{\epsilon}.$$
It follows that $u_{\epsilon}$ is uniformly bounded in $C^{\frac{q-n+2}{q-p}}(\overline{U_{\epsilon}})$. We can choose a subsequence of $u_{\epsilon}$ that uniformly converges to a limit $u\in C^0(\overline{U})$  which satisfies $u=0$ on $\partial U$ and
$$\|u\|_{C^{0}(\overline{U})} \geq c(p, q, n, \inf g) \left(|U|^{\ast}\right)^{\frac{1}{q-p}}.$$
According to Lemma 1.2.3 in \cite{Gut-16}, we know that $u$ is actually an Aleksandrov solution of \eqref{Eq-transfer-1}.

For any $\delta \in (0, \|u\|_{C^{0}(\overline{U})})$, let $U_{\delta}=\{x\in U: u(x)\leq -\delta \}$, which is convex with nonempty interior. Note that
$$-C_{\delta}\leq u\leq -\delta, \quad  c_{\delta} \leq |D u|^2+(x\cdot Du-u)^2 \leq C_{\delta} \quad \mbox{in} \quad U_{\delta},$$
if we choose $\delta$ small enough. Thus the Monge-Amp\`ere measure $M_{u}$, which is the weak limit of $\det D^2u_{\epsilon}$,
satisfies
$$0<c_{\delta} \leq M_{u} \leq C_{\delta}<\infty\quad \mbox{in} \quad U_{\delta}.$$
Therefore $u$ is strictly convex in $U_{\delta}$ and $u\in C^{1, \alpha}(U_{\delta})$ by Theorem 5.4.10 and Theorem 5.4.8 in \cite{Gut-16}.

For any $x_0\in \overline{U_{\delta}}$ and $p_0\in \partial u(x_0)$, we know that there exists a constant $t_0$ such that $\Sigma_{t_0}:=\{ u(x)<l_0(x)=u(x_0)+p_0\cdot(x-x_0)+t_0\} \subset \subset U$ by the similar method in the proof of Theorem 1.1 in \cite{Chen-Huang-19}.  Then by Pogorelov's interior estimates (Theorem 17.19 in \cite{GT01}), we know
$$(l_0-u)|D^2 u| \leq C(n, |u|_{C^{0,1}(\Sigma_{t_0})}, \delta) \quad \mbox{in} \quad \Sigma_{t_0}.$$
It implies $|D^2u|_{U_{\delta}} \leq C_{\delta}$ and the equation is uniformly elliptic in $U_{\delta}$. Using Evans-Krylov's estimates \cite{Ev82, Kr82}, we have
$$\|u\|_{C^{k, \alpha}(\overline{U_{\delta}})} \leq C(\delta, k).$$
We conclude $u\in C^{\infty}(U)$.
Moreover, it is easy to obtain the uniqueness of solution to  \eqref{Eq-transfer-1}  by the comparison principle.

\textbf{Step 3:} We show the optimality of the exponent $\frac{q-n+2}{q-p}$.

Indeed, for any $a\in (\frac{q-n+2}{q-p}, 1)$,  we choose $U$ and the function $w$ as in Lemma \ref{app-sup-1}. It follows that $w$ is a supersolution to the equation  \eqref{Eq-transfer-1}.
We show that $w\geq u$ in $U$. Note that $w=0\geq u$ on $\partial U$. If $w-u$ attains its minimum value on $\overline{U}$ at $y\in U$ with $w(y)<u(y)<0$, then $Dw(y)=Du(y)$ and $D^2 w(y)\geq D^2 u(y)$. It follows that at $y$
$$(-w)^{1-p}\left[|D w|^2+(x\cdot Dw-w)^2\right]^{\frac{q-n}{2}}\leq (-u)^{1-p}\left[|D u|^2+(x\cdot Du-u)^2\right]^{\frac{q-n}{2}},$$
which contradicts $w(y)<u(y)<0$.

For $x=(0, x_{n-1})\in U$, we have
\begin{equation}\label{app-the-for-2}
|u(x)|\geq |w(x)|=C\left(x_{n-1}^a- x_{n-1}\right)\geq \frac{C}{2}x_{n-1}^a=\frac{C}{2} \left(\mbox{dist}(x, \partial U) \right)^a
\end{equation}
by assuming  $x_{n-1}<\log_{\frac{1}{2}} (1-a)$, which implies the optimality of the exponent.
\end{proof}

\begin{proof}[Proof of Theorem \ref{th-main-2}]
The theorem can be easily obtained by   Theorem \ref{exist-results-p2} and Lemma \ref{app-lem1}.
\end{proof}

\section{Appendix}

\subsection{The Monge-Amp\`ere equation in Euclidean space}

We transfer the Monge-Amp\`ere equation \eqref{Eq} on $\Omega \subset S^{n-1}$ to a
Euclidean Monge-Amp\`ere equation on $U \subset \mathbb{R}^{n-1}$.
For $e \in S^{n-1}$, we consider the
restriction of a solution $h$ of \eqref{Eq}  to the hyperplane $e^{\bot}$ tangent to $S^{n-1}$ at $e$, i.e.
\begin{equation*}
u(x)=h(x+e).
\end{equation*}
We consider $\pi: e^{\bot}\rightarrow S^{n-1}$ defined by
\begin{equation*}
\pi(x)=\frac{1}{\sqrt{1+|x|^2}}(x+e).
\end{equation*}
Thus,
\begin{equation*}
u(x)=\sqrt{1+|x|^2}h(\pi(x)).
\end{equation*}
Let $\nabla$, $\overline{\nabla}$ and $D$ be the standard Levi-Civita connections in $S^{n-1}$,
$\mathbb{R}^n$, and $e^{\bot}=\mathbb{R}^{n-1}$.

\begin{lemma}\label{app-lem1}
The Dirichlet problem \eqref{Eq} of $h$ is equivalent to the following Dirichlet problem of $u$
\begin{equation*}
\left\{
\begin{aligned}
&\mathrm{det}(D^2u)=g(x)(-u)^{p-1}\Big[|Du|^2+(x\cdot Du-u)^2\Big]^{\frac{n-q}{2}}\quad \mbox{in}  \quad U\subset \mathbb{R}^{n-1},
&\\
&u=0 \quad \mbox{on}  \quad \partial U,
\end{aligned}
\right.
\end{equation*}
where $U=\pi^{-1}(\Omega)$ and
\begin{equation}\label{gg}
g(x)=f(\pi(x))(1+|x|^2)^{-\frac{n+p}{2}}.
\end{equation}
\end{lemma}

\begin{proof}
Note that
\begin{equation*}
tu(x)=h(tx+te).
\end{equation*}
Differentiating both sides of the above equation with $t$ and $x$ respectively, we obtain
\begin{equation*}
u(x)=\sum_{i=1}^{n}x^i\cdot\overline{\nabla}_{i}h+\overline{\nabla}_{n}h
\end{equation*}
and
\begin{equation*}
tD_iu(x)=t \overline{\nabla}_{i}h(tx+te).
\end{equation*}
Thus, we have by letting $t=\frac{1}{\sqrt{1+|x|^2}}$
\begin{equation*}
u(x)=\sum_{i=1}^{n}x^i\cdot\overline{\nabla}_{i}h(\pi(x))+\overline{\nabla}_{n}h(\pi(x))
\end{equation*}
and
\begin{equation*}
D_iu(x)=\overline{\nabla}_{i}h(\pi(x)).
\end{equation*}
Therefore (see also Page 500 in \cite{Ch-Yau-76}),
\begin{equation}\label{h-u}
|h(\pi(x))|^2+|\nabla h(\pi(x))|^2=|\overline{\nabla}h(\pi(x))|^2=|Du|^2+(x\cdot Du-u)^2.
\end{equation}
On the other hand, we have (see (2.4) in \cite{Ch-Yau-76})
\begin{equation*}
(1+|x|^2)^{\frac{n+1}{2}}\mathrm{det}(D^2 u(x))=\mathrm{det}(\nabla^2 h(\pi(x))+h(\pi(x)) I).
\end{equation*}
Thus,
\begin{equation*}
\mathrm{det}(D^2u)=f(\pi(x))(1+|x|^2)^{-\frac{n+p}{2}}(-u)^{p-1}\Big[|Du|^2+(x\cdot Du-u)^2\Big]^{\frac{n-q}{2}}.
\end{equation*}
\end{proof}

\subsection{The a priori estimates for solutions to the parabolic Monge-Amp\`ere equation}

Let $U$ be an open, bounded, smooth and strictly convex domain in $\mathbb{R}^n$. We denote by
\begin{eqnarray*}
\mathcal{C}_0=\{u \in C^{\infty}(\overline{U}): D^2 u>0, \ u|_{\partial U}=0\}.
\end{eqnarray*}
We consider the initial-boundary problem of the type
\begin{equation}\label{PMA}
\left\{
\begin{aligned}
&u_t-\log \det (D^2 u)=-g(x, u, Du) \quad \mbox{in}  \quad U \times (0, T],
&\\
&u=0 \quad \mbox{on}  \quad \partial U \times [0, T],
&\\&u=u_0 \quad \mbox{on}  \quad  U \times \{0\},
\end{aligned}
\right.
\end{equation}
where $g(x, u, Du)=\log f(x, u, Du)$ and $u_0 \in \mathcal{C}_0$ satisfies the compatibility condition
\begin{eqnarray*}
\det (D^2 u_0)=f(x, u_0, Du_0) \quad \mbox{on} \ \quad \partial U.
\end{eqnarray*}

Let $u \in C^4(U \times (0, T))\cap C^2(\overline{U} \times [0, T])$ be a solution to \eqref{PMA}, and suppose further that
\begin{eqnarray}\label{LGC0}
-K\leq u(x, t)<0, \quad \forall \ (x, t) \in \overline{U} \times [0, T]
\end{eqnarray}
and
\begin{eqnarray}\label{LGC}
0<f(x, u(x), p)\leq C(1+p^2)^{\frac{n}{2}},
\quad \forall \ x \in \overline{U},
\end{eqnarray}
where $C$ is a positive constant depending only on $K$.

Now, we will establish the a priori estimates for solutions to the initial-boundary problem \eqref{PMA}.

\begin{lemma}\label{A-gra}
Let $u \in C^4(U \times (0, T))\cap C^2(\overline{U} \times [0, T])$ be a solution to \eqref{PMA} satisfying
the assumptions \eqref{LGC0} and \eqref{LGC}. Then, we have
\begin{eqnarray*}
|D u|(x, t)\leq C, \quad \forall \ (x, t) \in \overline{U} \times [0, T].
\end{eqnarray*}
\end{lemma}

\begin{proof}
Using the condition \eqref{LGC} and following the same argument in section 7 in \cite{CNS1}, there exists a
convex subsolution $\underline{u} \in C^2(\overline{U})$
\begin{equation*}
\left\{
\begin{aligned}
&\det (D^2 \underline{u})\geq C (1+|D\underline{u}|^2)^{\frac{n}{2}} \quad \mbox{in} \quad  U ,&\\
&\underline{u}=0, \quad \mbox{on} \quad  \partial U.
\end{aligned}
\right.
\end{equation*}
Set $\underline{v}=\mu \underline{u}+u_0$. For large $\mu$, it is easy to show that $\underline{v}$
also satisfies the above inequality with the same boundary value.
Since
\begin{eqnarray*}
\underline{v}=0 \ \mbox{on} \ \partial U \times [0, T), \quad \underline{v}\leq u_0 \ \mbox{on} \ U \times \{0\},
\end{eqnarray*}
we have by maximum principle $\underline{v}\leq u$ in $\overline{U}\times [0, T]$,
it follows that
\begin{eqnarray*}
0\leq \frac{\partial u}{\partial \nu}\leq \frac{\partial \underline{v}}{\partial \nu} \quad \mbox{on} \quad  \partial U,
\end{eqnarray*}
where $\nu$ is the unit outer vector of $\partial U$. Due to the convexity of $u$, we have
\begin{eqnarray*}
|D u|_{C^0(\overline{U}\times [0, T])}\leq \Big|\frac{\partial \underline{v}}{\partial \nu}\Big|_{C^0(\partial U)},
\end{eqnarray*}
which completes the proof.
\end{proof}

Based on the above gradient estimate, we can follow the same arguments in Step 1 and Step 3
in Appendix \cite{Tso90} to obtain the a priori estimates for $u_t$.
Then, we follow almost the same argument in Section 7 in \cite{CNS1} to get the global second order
estimates of $u$ for the variable $x$.

\begin{lemma}\label{D2-i}
Let $u \in C^4(U \times (0, T))\cap C^2(\overline{U} \times [0, T])$ be a solution to \eqref{PMA} satisfying
the assumptions \eqref{LGC0} and \eqref{LGC}. Then, we have
\begin{eqnarray*}
|u_t(x, t)|+|D^2 u(x, t)|\leq C, \quad \forall \ (x, t) \in \overline{U} \times [0, T].
\end{eqnarray*}
\end{lemma}

\subsection{The a priori estimates for solutions to the Monge-Amp\`ere equation}

We consider the a priori estimates of to solutions of the Dirichlet problem
\begin{equation}\label{D-U}
\left\{
\begin{aligned}
&\det(D^2 u)=f(x, u, Du) \quad \mbox{in} \quad  U\subset \mathbb{R}^n ,&\\
&u=0, \quad \mbox{on} \quad  \partial U.
\end{aligned}
\right.
\end{equation}
Let $u \in C^4(U)\cap C^2(\overline{U})$ be a solution to \eqref{PMA}, and suppose further that
\begin{eqnarray*}
-K\leq u(x)<0, \quad \forall \ x \in U
\end{eqnarray*}
and
\begin{eqnarray*}
0<f(x, u(x), D^2u(x))\leq C(1+|D^2u(x)|^2)^{\frac{n}{2}},
\quad \forall \ x \in U,
\end{eqnarray*}
where $C$ is a positive constant depending only on $K$.
\begin{lemma}\label{D-est}
We have

(1) The gradient estimate
\begin{eqnarray*}
|D u|(x)\leq C, \quad \forall \ x \in \overline{U},
\end{eqnarray*}
where $C$ is a positive constant depending only on $K$.

(2)The high order estimates
\begin{eqnarray*}
\parallel u\parallel_{C^{k, \alpha}(U^{\prime})}\leq C, \quad \forall  \  U^{\prime} \subset \subset U,
\end{eqnarray*}
where $C$ is a positive constant depending only on $K$,
$d(U^{\prime}, \partial U)$, $\inf_{U^{\prime}} f$, the bounds on $f$
and its derivatives on $U^{\prime}$.

(3) If $f(x, u(x), Du(x))\geq \eta>0$, we have
\begin{eqnarray*}
\parallel u\parallel_{C^{k, \alpha}(\overline{U})}\leq C,
\end{eqnarray*}
where $C$ is a positive constant depending only on $K$ and $f$.
\end{lemma}

\begin{proof}
The gradient estimate can be deduced by Lemma \ref{A-gra} in Appendix. By Pogorelov's
interior estimates \cite{GT01, Po76} and Evans-Krylov estimates \cite{Ev82, Kr82},
we have the interior high order estimates. The global high order estimates from Theorem 7 in \cite{CNS1}.
\end{proof}


\bigskip

\bigskip

\begin{thebibliography}{50}
\setlength{\itemsep}{-0pt} \small

\bibitem{AYY} W. Ai, Y. L. Yang, D. P. Ye, The $L_p$ dual Minkowski problem
for unbounded closed convex sets, arXiv preprint, arXiv:2404.09804v1.

\bibitem{Ba-61}
I. J. Bakelman,  A variational problem related to the Monge-Amp\`ere equation, Dokl. Akad. Nauk SSSR, 141 (1961), 1011-1014.

\bibitem{Ba-83}
I. J. Bakelman,  Variational problems and elliptic Monge-Amp\`ere equations,
 J. Differential Geom., 18 (1983), no. 4, 669-699 (1984).

\bibitem{BIS19}
P. Bryan, M. N. Ivaki, J. Scheuer, A unified flow approach to
smooth, even $L_p$-Minkowski problems, Anal. PDE, 12 (2019), 259-280.

\bibitem{CHZ19}
C. Chen, Y. Huang, Y. M. Zhao, Smooth solutions to the $L_p$ dual
Minkowski problem, Math. Ann., 373 (2019), 953-976.

\bibitem{CCL21} H. Chen, S. Chen, Q. R. Li, Variations of a class of Monge-Amp\`ere-type functionals and their
applications, Anal. PDE, 14 (2021), pp. 689-716.

\bibitem{CLLN22}
L. Chen, Y. N. Liu, J. Lu, N. Xiang,
Existence of smooth even solutions to the dual
Orlicz-Minkowski problem, J. Geom. Anal., 32(2) (2022), Paper No. 40, 25 pp.

\bibitem{CTWX22}
L. Chen, Q. Tu, D. Wu, N. Xiang,
Anisotropic Gauss curvature flows and their associated Dual Orlicz-Minkowski problems, Proc. Roy. Soc. Edinburgh Sect. A, 152 (2022), 148-162.

\bibitem{CWX22} L. Chen, D. Wu, N. Xiang, Smooth solutions to the Gauss image problem, Pacific J. Math., 317(2) (2022), 275-295.

\bibitem{CNS1} L. Caffarelli, L. Nirenberg, J. Spruck, The Dirichlet problem for nonlinear second-order elliptic
equations. I. Monge-Amp\`ere equation, Comm. Pure Appl. Math. 37 (1984), no. 3, 369-402.

\bibitem{Chen-Huang-19}
H. Chen, G. G. Huang,  Existence and regularity of the solutions of some singular Monge-Amp\`ere equations, J. Differential Equations, 267 (2019), no. 2, 866-878.

\bibitem{Chen-L19} H. Chen, Q. R. Li,
The $L_p$ dual Minkowski problems and related parabolic
flows, J. Funct. Anal., 281 (2021), 109139.

\bibitem{Ch-Yau-77}
S. Y. Cheng, S. T. Yau,  On the regularity
of the Monge-Amp\`ere equation $\det (\frac{\partial^2 u}{\partial x^i \partial x^j})=F(x, u)$, Comm. Pure Appl. Math., 30 (1977), no.1, 41-68.

\bibitem{Ch-Yau-76}
S. Y. Cheng, S. T. Yau,  On the regularity of the solution of the $n$-dimensional
Minkowski problem, Comm. Pure Appl. Math.,  29 (1976), no. 5, 495-516.

\bibitem{CW95}
K. S.  Chou, X. J. Wang,
Minkowski problems for complete noncompact convex hypersurfaces, Topol.
Methods Nonlinear Anal. 6 (1995), 151-162.

\bibitem{CW}
K. S.  Chou and X. J Wang,
The $L_p$-Minkowski problem and the Minkowski problem in centroaffine geometry,
 Adv. Math.,  205 (2006),  33-83.

\bibitem{CW00}
K. S. Chou, X. J. Wang, A logarithmic Gauss curvature flow and
the Minkowski problem, Ann. Inst. H. Poincar\'e Anal. Non Lin\'eaire, 17
(2000), 733-751.

\bibitem{Ev82} L. Evans, Classical solutions of fully nonlinear, convex, second-order elliptic equations.
Comm. Pure Appl. Math., 35 (1982), no. 3, 333-363.

\bibitem{GT01} D. Gilbarg, N. Trudinger, Elliptic partial differential equations of second order, Reprint of the 1998 edition. Classics in Mathematics. Springer-Verlag, Berlin, 2001. xiv+517 pp.

\bibitem{Gut-16}
C. E. Guti\'errez, The Monge-Amp\`ere equation. Second edition,  Birkha\"user, Boston, 2016.

\bibitem{Huang16} Y. Huang, E. Lutwak, D. Yang, G. Y. Zhang,
Geometric measures in the dual Brunn-Minkowski theory and
their associated Minkowski problems, Acta Math., 216
(2016), 325-388.

\bibitem{HZ}
Y. Huang, Y. M. Zhao, On the $L_p$ dual Minkowski problem, Adv. Math., 332 (2018), 57-84.

\bibitem{HL21}
Y. Huang, J. K. Liu,
Noncompact $L_p$-Minkowski problems, Indiana Univ. Math. J. 70 (2021), 855-880.

\bibitem{JWW21}
Y. S. Jiang, Z. Wang, Y. H. Wu, Multiple solutions of the planar $L_p$ dual Minkowski problem, Calc.
Var. Partial Differential Equations, 60 (2021), pp. Paper No. 89, 16.

\bibitem{JWW23}
Y. S. Jiang, Z. Wang, Y. H. Wu, Variational analysis of the planar $L_p$ dual Minkowski problem, Math. Ann., 386 (2023), 1201-1235.

\bibitem{LLL22}
Q. R. Li, J. K. Liu, J. Lu, Nonuniqueness of solutions to the $L_p$ dual Minkowski problem, Int. Math. Res. Not. IMRN, (2022), 9114-9150.

\bibitem{LSW20}
Q. R. Li, W. M. Sheng, X. J. Wang, Flow by Gauss curvature to the
Aleksandrov and dual Minkowski problems, J. Eur. Math. Soc., 22
(2020), 893-923.

\bibitem{LL20}
Y. N. Liu, J. Lu, A flow method for the dual Orlicz-Minkowski
problem, Trans. Amer. Math. Soc., 373 (2020), 5833--5853.


\bibitem{KS.IANSSM.44-1980.161}
 N. V. Krylov, M. V. Safonov, A property of the solutions of
parabolic equations with measurable coefficients, Izv. Akad. Nauk SSSR Ser.
Mat., 44 (1980), 161-175, 239.

\bibitem{Kr82}
 N. V. Krylov, Bounded inhomogeneous elliptic and parabolic equations,
Izv. Akad. Nauk SSSR Ser. Mat. 46 (1982), no. 3, 487-523, 670.

\bibitem{LYZ} N. Li, D. P. Ye, B. C. Zhu, The dual Minkowski problem
for unbounded closed convex sets, Math. Anna., 388 (2023), 2001-2039.

\bibitem{Li98} G. Lieberman, Second order parabolic differential equations,
Woeld Scientific Publishing, Singapore, 1998. xiv+439 pp.

\bibitem{LNQ-1}
N. Q. Le,  Optimal boundary regularity for some singular Monge-Amp\`ere equations
on bounded convex domains, Discrete Contin. Dyn. Syst., 42 (2022), 2199-2214.

\bibitem{LNQ-2}
N. Q. Le, Remarks on sharp boundary estimates for
singular and degenerate Monge-Amp\`ere equations, Commun. Pure Appl. Anal., 22 (2023), 1701-1720.

\bibitem{LS17} N. Q. Le, O. Savin, Schauder estimates for degenerate Monge-Amp\`ere equations and smoothness of the eigenfunctions, Invent. Math., 207 (2017), no. 1, 389-423.

\bibitem{LE1}
E. Lutwak,
The Brunn-Minkowski-Firey theory. I. Mixed volumes and the Minkowski problem,
J. Differential Geom.,  38 (1993),   131-150.

\bibitem{LYZ-18}
E. Lutwak, D. Yang,  G. Y. Zhang, $L_p$ dual curvature measures, Adv. Math., 329 (2018), 85–132.

\bibitem{P071}
A. Pogorelov,  The regularity of the generalized solutions of the equation $\det(\frac{\partial^2 u}{\partial x_i \partial x_j})=\varphi(x_1, x_2, \cdots, x_n)>0$,  (Russian) Dokl. Akad. Nauk SSSR 200 (1971), 534-537.

\bibitem{Po76}
A. Pogorelov, The Minkowski multidimensional problem, Translated from the Russian
by Vladimir Oliker. Introduction by Louis Nirenberg. Scripta Series in Mathematics. V. H. Winston and Sons,
Washington, D.C.; Halsted Press [John Wiley and Sons], New York-Toronto-London, 1978. 106 pp.

\bibitem{Po80} A. Pogorelov, An analogue of the Minkowski problem for infinite complete convex hypersurfaces.
Dokl. Akad. Nauk SSSR 250 (1980), 553-556.

\bibitem{Sa14} O. Savin, A localization theorem and boundary regularity for a class of degenerate Monge
Amp\`ere equations. J. Differential Equations, 256 (2014), no. 2, 327-388.

\bibitem{Sch13} R. Schneider, Convex bodies: the Brunn-Minkowski theory,
Second edition, No. 151. Cambr. Univ. Press, 2013.

\bibitem{Sch-18}
R. Schneider, A Brunn-Minkowski theory for coconvex sets of finite volume, Adv. Math., 332 (2018), 199-234.

\bibitem{Sch-21}
R. Schneider, Minkowski type theorems for convex sets in cones, Acta Math.  Hung., 164 (2021), 282-295.


\bibitem{TY}
F. Tong, S. T. Yau, Generalized Monge-Amp\`ere functionals and related variational problems, 2023, arXiv preprint,  arXiv:2306.01636v1.

\bibitem{Tso90} K. Tso, On a real Monge-Amp\`ere functional, Invent. math., 101 (1990), 425-448.

\bibitem{Ur-1} J. Urbas, An expansion of convex hypersurfaces, J. Differential Geom., 33 (1991), 91-125.

\bibitem{Ur84} J. Urbas, The equation of prescribed Gauss curvature without boundary conditions. J. Differential Geom.,
20 (1984), 311-327.


\bibitem{YYZ}
J. Yang, D. P. Ye, B. C. Zhu, On the $L_p$ Brunn-Minkowski theory and the $L_p$ Minkowski problem for C-coconvex sets,  Int. Math. Res. Not. IMRN, 7 (2023), 6252-6290.
\end{thebibliography}
\end{document}